\documentclass[12pt,english]{article}
\usepackage[T1]{fontenc}
\usepackage[latin9]{inputenc}
\usepackage{geometry}
\usepackage{babel}
\usepackage{mathtools}
\usepackage{amsmath}
\usepackage{amsthm}
\usepackage{amssymb}
\usepackage[unicode=true,pdfusetitle,
 bookmarks=true,bookmarksnumbered=false,bookmarksopen=false,
 breaklinks=false,pdfborder={0 0 1},backref=false,colorlinks=false]
 {hyperref}

\makeatletter
\numberwithin{equation}{section}
\numberwithin{figure}{section}
\theoremstyle{plain}
\newtheorem{thm}{\protect\theoremname}
\theoremstyle{remark}
\newtheorem{rem}[thm]{\protect\remarkname}
\theoremstyle{definition}
\newtheorem{defn}[thm]{\protect\definitionname}
\theoremstyle{plain}
\newtheorem{cor}[thm]{\protect\corollaryname}
\theoremstyle{definition}
\newtheorem{example}[thm]{\protect\examplename}
\theoremstyle{plain}
\newtheorem{lem}[thm]{\protect\lemmaname}
\theoremstyle{remark}
\newtheorem{notation}[thm]{\protect\notationname}
\theoremstyle{plain}
\newtheorem{prop}[thm]{\protect\propositionname}
\theoremstyle{remark}
\newtheorem{claim}[thm]{\protect\claimname}

\date{}
\usepackage{tikz-cd}
\usepackage{wasysym}
\usepackage{MnSymbol}
\usepackage[tt=false]{libertine}
\usepackage[parfill]{parskip}
\usepackage{enumitem}
\setlist[itemize]{noitemsep,topsep=5pt}

\usepackage{titlesec}
\titleformat{\section}{\large\bfseries\filleft}{\thesection}{1em}{}[{\titlerule[0.8pt]}]

\renewcommand\labelenumi{(\roman{enumi})}
\renewcommand\theenumi\labelenumi

\DeclareMathOperator{\Spec}{Spec}

\DeclareMathOperator{\Pic}{Pic}

\DeclareMathOperator{\im}{im}

\DeclareMathOperator{\rk}{rk}

\DeclareMathOperator{\Id}{Id}

\DeclareMathOperator{\Mor}{Mor}
\DeclareMathOperator{\Res}{Res}
\DeclareMathOperator{\divisor}{div}

\let\oldtheorem\thm
\renewcommand{\thm}{\oldtheorem\normalfont}
\let\oldprop\prop
\renewcommand{\prop}{\oldprop\normalfont}
\let\oldcor\cor
\renewcommand{\cor}{\oldcor\normalfont}
\let\oldlem\lem
\renewcommand{\lem}{\oldlem\normalfont}
\newenvironment{ack}{\textit{Acknowledgements.}}{}
\makeatother

\providecommand{\claimname}{Claim}
\providecommand{\corollaryname}{Corollary}
\providecommand{\definitionname}{Definition}
\providecommand{\examplename}{Example}
\providecommand{\lemmaname}{Lemma}
\providecommand{\notationname}{Notation}
\providecommand{\propositionname}{Proposition}
\providecommand{\remarkname}{Remark}
\providecommand{\theoremname}{Theorem}

\begin{document}
\global\long\def\A{\mathbb{A}}%

\global\long\def\C{\mathbb{C}}%

\global\long\def\E{\mathbb{E}}%

\global\long\def\F{\mathbb{F}}%

\global\long\def\G{\mathbb{G}}%

\global\long\def\H{\mathbb{H}}%

\global\long\def\N{\mathbb{N}}%

\global\long\def\P{\mathbb{P}}%

\global\long\def\Q{\mathbb{Q}}%

\global\long\def\R{\mathbb{R}}%

\global\long\def\O{\mathcal{O}}%

\global\long\def\Z{\mathbb{Z}}%

\global\long\def\ep{\varepsilon}%

\global\long\def\laurent#1{(\!(#1)\!)}%

\global\long\def\wangle#1{\left\langle #1\right\rangle }%

\global\long\def\ol#1{\overline{#1}}%

\global\long\def\mf#1{\mathfrak{#1}}%

\global\long\def\mc#1{\mathcal{#1}}%

\global\long\def\norm#1{\left\Vert #1\right\Vert }%

\global\long\def\et{\textup{ét}}%

\global\long\def\Et{\textup{Ét}}%

\title{A higher genus circle method and an application to geometric Manin's
conjecture}
\author{Matthew Hase-Liu\thanks{Department of Mathematics, Columbia University, New York, NY}\thanks{Email address: m.hase-liu@columbia.edu}}
\maketitle
\begin{abstract}
Browning and Vishe used the circle method to show the moduli space
of rational curves on smooth hypersurfaces of low degree is irreducible
and of the expected dimension. We reinterpret the circle method geometrically
and prove a generalization for higher genus smooth projective curves.
In particular, we explain how the geometry of numbers can be understood
via the Beauville--Laszlo theorem in terms of vector bundles on curves
and their slopes, allowing us to prove a higher genus variant of Davenport's
shrinking lemma. As a corollary, we apply this result  to show the
Fujita invariant of any proper subvariety of a smooth hypersurface
of low degree is less than 1. 
\end{abstract}
\tableofcontents{}

\section{Introduction}

Let $X\subset\P_{\C}^{n}$ be a smooth hypersurface. A natural object
to study is the moduli space of curves on $X$. More precisely, for
a fixed smooth projective curve $C$ and a non-negative integer $e$,
there is a quasi-projective variety $\Mor_{e}\left(C,X\right)$ parametrizing
degree $e$ morphisms $C\to X$ (see Chapter 2 of \cite{debarre01}
for some basic properties, such as the expected dimension). 

While most results in the literature are about rational curves on
generic $X$ (\cite{HarrisRothStarrRatCurvesI,beheshtikumar13,RiedlYang}),
there are few that extend to arbitrary $X$ (\cite{BrowningVisheRatCurves,coskunstarr09}),
and even fewer for the case of curves of genus larger than 0 mapping
to arbitrary $X$. Browning and Vishe show in \cite{BrowningVisheRatCurves}
that for $X$ of degree $d\ge3$ and $n\ge2^{d-1}\left(5d-4\right)$
(later improved to $n\ge2^{d}\left(d-1\right)$ by Bilu and Browning
in \cite{bilu2023motivic}) the moduli space $\Mor_{e}\left(\P^{1},X\right)$
is irreducible and has the expected dimension $\left(n+1\right)\left(e+1\right)-de-2$.
In our paper, we extend Browning and Vishe's results to smooth projective
curves of higher genus. 

More precisely, our main theorem is as follows:
\begin{thm}
\label{thm:main}Let $C$ be a smooth projective curve of genus $g\ge1$
and $X\subset\P^{n}$ a smooth hypersurface of degree $d$, both defined
over $\C$. Suppose the following holds: 

\[
n\ge\begin{cases}
5 & \text{and }d=2,g=1,e\ge14,\text{ or}\\
5 & \text{and }d=2,g\ge2,e\ge\frac{35}{2}g-\frac{15}{2},\text{ or}\\
17 & \text{and }d=3,g=1,e\ge45,\text{ or}\\
17 & \text{and }d=3,g\ge2,e\ge\frac{170}{3}g+\frac{32}{3},\text{ or}\\
2^{d}\left(d-1\right)+1 & \text{and }d\ge4,g=1,e\ge2^{d}\left(d-1\right)^{2}+1,\text{ or}\\
2^{d}\left(d-1\right)+1 & \text{and }d\ge4,g\ge2,e\ge2^{d-1}\left(d-1\right)^{2}\left(3g+1\right)+d\left\lceil \frac{3g-1}{2}\right\rceil +\left\lfloor \frac{g-1}{2}\right\rfloor .
\end{cases}
\]

Then, the mapping space $\Mor_{e}\left(C,X\right)$ is irreducible
and has the expected dimension $(n+1)(e+1-g)-de-2+2g$.
\end{thm}

\begin{rem}
We omit the case $d=1$ because $\Mor_{e}\left(C,X\right)$ is already
known to be smooth for $e\ge2g-1$ (see Lemma 7 of \cite{non-smooth}).
However, the methods of this paper can still be applied to show $\Mor_{e}\left(C,X\right)$
is irreducible and has the expected dimension for $e\ge4g$. 

Also, the case $d=2$ of quadrics in the theorem is not optimal: Pasquier
and Perrin show in \cite{ellipticcurveshomogeneous} that it suffices
to take $e\ge3$ for any $n$ (see also \cite{ballico}). We nonetheless
include this case in the statement for the sake of uniformity.
\end{rem}

\begin{rem}
The theorem above also holds for when $C$ is allowed to vary in moduli.
Let $\mathcal{M}_{g,0}\left(X,e\right)$ denote the moduli space of
degree $e$ maps from a genus $g$ smooth projective curve to $X$
a smooth degree $d$ hypersurface in $\P_{\C}^{n}$. Suppose $n,d,e,$
and $g$ satisfy the same conditions as in the theorem. Then, $\mathcal{M}_{g,0}\left(X,e\right)$
is irreducible and has the expected dimension $(n+1)(e+1-g)-de-2+2g+\left(3g-3\right)=(n+1)(e+1-g)-de-5+5g.$ 

Indeed, recall the canonical forgetful map $\mathcal{M}_{g,0}\left(X,e\right)\to\mathcal{M}_{g}$
that sends a morphism $C\to X$ to $C$, whose fibers are $\Mor_{e}\left(C,X\right)$.
Since $\dim\mathcal{M}_{g}=3g-3$, the dimension claim is clear, i.e.
$\mathcal{M}_{g,0}\left(X,e\right)$ is a local complete intersection
stack. Since $\mathcal{M}_{g}$ is smooth, a stacky version of miracle
flatness then implies the forgetful map is flat, which implies $\mathcal{M}_{g,0}\left(X,e\right)$
is irreducible.
\end{rem}

As a corollary of Theorem \ref{thm:main}, we can control the Fujita
invariants of subvarieties of low degree smooth hypersurfaces, which
is of interest in the context of geometric Manin's conjecture. For
an excellent introduction to geometric Manin's conjecture, see Tanimoto's
survey paper \cite{tanimoto21}. 

Recall the definition of the Fujita invariant:
\begin{defn}
\label{def:fujita}Let $X$ be a smooth projective variety and $L$
a big and nef $\Q$-divisor on $X$. The \textit{Fujita invariant}\textbf{
}is 
\[
a\left(X,L\right)\coloneqq\min\left\{ t\in\R:tL+K_{X}\text{ is pseudo-effective}\right\} .
\]
\end{defn}

Here, the sum $tL+K_{X}$ is taken inside the $N^{1}(X)\otimes\R$,
where $N^{1}(X)$ is the space of Cartier divisors of $X$ up to numerical
equivalence, and a divisor is pseudo-effective if its class is in
the closure of the convex cone spanned by effective $\R$-divisors.

In particular, if $X$ is Fano and $K_{X}$ is the canonical divisor,
then $a\left(X,-K_{X}\right)=1$. 

\begin{cor}
\label{cor:anumber}Let $X$ be a smooth hypersurface in $\P_{\C}^{n}$
of degree $d\ge2$ such that $n\ge2^{d}(d-1)+1$. Then, if $V$ is
a proper subvariety of $X$, we have $a\left(V,-K_{X}|_{V}\right)<1=a\left(X,-K_{X}\right)$. 
\end{cor}

\begin{rem}
In fact, combining Corollary \ref{cor:anumber} with a variant of
the proof of Proposition 6 in \cite{haseliu2025conversegeometricmaninsconjecture}
even shows there are no generically finite and non-birational maps
$f\colon V\to X$ with $a\left(V,-K_{X}|_{V}\right)\ge1$. 
\end{rem}

Theorem 1.1 of \cite{Lehmann_Tanimoto_2019} (which is a combination
of Theorem 4.8 of \cite{LTT18}, Theorem 1.1 of \cite{haconjiang17},
and Birkar's resolution of the Borisov--Alexeev--Borisov conjecture
in \cite{Bir2}) tells us that for smooth projective weak Fano variety,
there is a proper closed subset of $X$ that is the closure of all
subvarieties $V$ that have Fujita invariant $a\left(V,-K_{X}|_{V}\right)$
 larger than $a\left(X,-K_{X}\right)$. So, in our situation where
$X$ is a smooth hypersurface of low degree, we have shown that this
exceptional subset is actually empty. The best known bounds for the
Fujita invariants of subvarieties of hypersurfaces are given by Examples
4.8 and 4.9 of the same paper, which prove that subvarieties of cubic
hypersurfaces (resp. quartic hypersurfaces) have Fujita invariants
at most (but not strictly less than) 1, assuming $n\ge3$ (resp. $n\ge5$).
In particular, our bound is the first to apply to hypersurfaces of
degree greater than 4. Moreover, our corollary about Fujita invariants,
in turn, implies every irreducible component of $\Mor\left(\P^{1},X\right)$
has the expected dimension by Theorem 4.6, again from the same paper.

We've shown something stronger in the corollary above, namely a strict
inequality. For instance, using Theorem 1.2 of \cite{nonfreecurves},
this means there is a constant $\xi=\xi(n,g)$ such that for any irreducible
component $M\subset\Mor\left(C,X\right)$ parametrizing non-free maps
of anti-canonical degree at least $\xi$, the evaluation map from
the normalization of the universal family over $M$ to $X$ is necessarily
dominant. 

Our proof of the main theorem closely follows the strategy of Browning
and Vishe, which in turn seems to have been considered in Pugin's
thesis \cite{Pugin_Thesis} and a suggestion of Ellenberg and Venkatesh,
i.e. using the circle method to study the moduli space of curves on
varieties. 

Let us briefly review this approach. The first step is to use a well-known
general lower bound on $\dim\Mor_{e}\left(C,X\right)$ and an easy
lemma to show the space is also non-empty. To obtain the upper bound,
we relate the situation over characteristic 0 with that over finite
characteristic. More specifically, we ``spread out'' to bound the
dimension of a certain generic fiber with that of any fiber above
a closed point (the generic fiber corresponding to the characteristic
0 case). Passing to the finite field analogue of the problem, the
Lang--Weil bounds imply that it suffices to obtain good bounds on
point-counts of our moduli space. To do so, we express the point-counts
as exponential sums over a ``circle'' that we subdivide into major
arcs and minor arcs. 

The major arcs contribute most to the sum, and we can check this without
too much difficulty via a direct computation. For the minor arcs,
we require good bounds to ensure they contribute very little to the
sum, and Browning and Vishe approach this by using standard techniques
from analytic number theory, such as Weyl differencing and Davenport's
shrinking lemma (which in turn uses facts from the geometry of numbers). 

So, how do we generalize their strategy to arbitrary smooth projective
curves? In our paper, we reinterpret many of their arithmetically-inspired
constructions geometrically. For instance, to define the ``circle''
for the exponential sum, Browning and Vishe choose a place at $\infty$
of $\P_{\F_{q}}^{1}$, and then set $\mathbb{T}=\left\{ \alpha\in K_{\infty}:\left|\alpha\right|<1\right\} $,
where $K_{\infty}$ is the completion of $\F_{q}(t)$ at the prime
$t^{-1}$ and $\left|\cdot\right|$ is the associated absolute value.
Instead, our method does not choose a specific place and defines the
``circle'' as linear functionals on the global sections of a certain
line bundle. 

Davenport's shrinking lemma and the geometry of numbers are also interpreted
geometrically by use of the Beauville--Laszlo theorem from \cite{BL95}
(although in our case where everything is Noetherian it is also simply
a consequence of faithfully flat descent). In particular, a lattice
over an affine curve and a choice of norm at $\infty$ (as is the
usual situation in the geometry of numbers) is equivalent to a vector
bundle on the entire curve---the Beauville--Laszlo theorem says
a vector bundle $V$ on a smooth projective curve $C$ is the same
as specifying a vector bundle on $C$ minus a point $\infty$ (i.e.
a lattice), a vector space over the completion of the function field
at $\infty$, and an isomorphism where the two agree. The geometric
analogue of successive minima (of a lattice) can then be understood
as the slopes arising from the Harder--Narasimhan filtration of the
associated vector bundle, and the shrinking lemma is proved by an
application of Riemann-Roch.

For the application to Fujita invariants of subvarieties, the general
strategy is to proceed by contradiction and assume a subvariety $V$
with large Fujita invariant exists. Then, after passing to finite
characteristic, we show there exists a smooth projective curve $C$
such that the mapping space of $C$ to $V$ is larger than the mapping
space of $C$ to $X$, which is impossible. The key point is that
we always have lower bounds for the dimension of the mapping space,
as well as upper bounds when the target is a low degree smooth hypersurface. 

If $Y\to V$ is a resolution of singularities, the assumption on large
Fujita invariant implies $K_{Y}-K_{X}|_{Y}$ is not pseudo-effective.
We then appeal to Theorem 0.2 of \cite{BDPP}, which implies for a
smooth projective variety $Y$ over $\C$ that if a line bundle $L$
is not pseudo-effective, then for a general point in $Y$, we can
find a curve passing through $Y$ such that the intersection of $L$
and the curve is negative (in particular, if $L=K_{Y}$, then by applying
a bend-and-break argument, this even shows $Y$ is uniruled). Theorem
0.2 applied to the line bundle $K_{Y}-K_{X}|_{Y}$ produces a curve
that intersects negatively with $K_{Y}-K_{X}|_{Y}$, but we have no
control over its degree. By passing to finite characteristic, replacing
the curve with an Artin--Schreier cover to increase genus, and increasing
the degree without changing genus by using the Frobenius (as done
in the bend-and-break lemmas), we can find a curve $C$ that gives
us our desired contradiction. 

As a final remark, we give an example of low degree maps where the
conclusion of Theorem \ref{thm:main} fails, showing that some lower
bound on the degree as in the hypothesis of Theorem \ref{thm:main}
is necessary.
\begin{example}
By \cite{BrowningVisheRatCurves}, we know for $d\ge3$ and $e\ge1$,
\[
\dim\Mor_{e}\left(\P^{1},X\right)=(n+1)(e+1)-de-2.
\]
Now, take a hyperelliptic cover $C\to\P^{1}$ such that $C$ has genus
$g$. Then, 
\[
\dim\Mor_{2e}\left(C,X\right)\ge\dim\Mor_{e}\left(\P^{1},X\right)=(n+1)(e+1)-de-2.
\]
Then, for large $g$, $\Mor_{2e}\left(C,X\right)$ cannot have the
correct expected dimension. Indeed, for 
\[
g>\frac{\left(n+1-d\right)e}{n-1},
\]
we see that 
\[
(n+1)(e+1)-de-2>(n+1)(2e+1-g)-2de-2+2g.
\]
\end{example}

\begin{ack} I would like to thank my advisor Will Sawin and Eric
Riedl for suggesting this problem; I'm especially grateful to Will
for his tremendous and invaluable guidance throughout this project.
I would also like to thank Tim Browning, Nathan Chen, Junxian Li,
Amal Mattoo, Morena Porzio, Akash Sengupta, and the anonymous referee
for their interest and many helpful comments.

The author was partially supported by National Science Foundation
Grant Number DGE-2036197.\end{ack}

\section{\label{sec:A-lower-bound}A lower bound on $\dim\Mor_{e}\left(C,X\right)$}

In this section, fix a smooth projective curve $C$ of genus $g\ge1$,
a smooth hypersurface $X\subset\P^{n}$ of degree $d$, both defined
over $\C$. We first claim $\Mor_{e}\left(C,X\right)$ is non-empty
for all cases of interest in the statement of Theorem \ref{thm:main}. 
\begin{lem}
\label{lem:existence}Suppose $d\le2n-3$ and $e\ge\max\left(2g,g+3\right)$.
Then there is a degree $e$ morphism from $C$ to $X$.
\begin{proof}
Since $d\le2n-3,$ by Theorem 8 of \cite{barthvan78}, the Fano variety
of lines on $X$ is non-empty, i.e. there is at least one line $\P^{1}\hookrightarrow X$.
Hence, it suffices to find a morphism $C\to\P^{1}$ of degree $e,$
since the composition $C\to\P^{1}\hookrightarrow X$ is then of degree
$e$. 

Let $x$ be a closed point of $C$ and consider $\mathcal{O}\left(ex\right)$.
By Riemann--Roch, we have $h^{0}\left(C,\mathcal{O}\left(ex\right)\right)=e-g+1>h^{0}\left(C,\mathcal{O}\left((e-1)x\right)\right)=e-g\ge\max\left(g,3\right)\ge1$
by assumption. Hence we can find a section $s\in H^{0}\left(C,\mathcal{O}(ex)\right)\backslash H^{0}\left(C,\mathcal{O}((e-1)x)\right)$,
so that $[1\colon s]$ defines a rational map (and hence necessarily
a morphism) from $C\to\P^{1}$ of degree $e$, as desired.
\end{proof}
\end{lem}

With the lemma above, the lower bound on dimension of $\Mor_{e}\left(C,X\right)$
is standard (for instance, see Chapter 2 of \cite{debarre01}). Let
$f$ be a degree $e$ morphism $C\to X$. Then, we have
\begin{align*}
\dim\Mor_{e}\left(C,X\right) & \ge-K_{X}\cdot f_{*}C+\left(1-g\right)\left(n-1\right)\\
 & =\left(n+1-d\right)e+\left(1-g\right)\left(n-1\right)\\
 & =(n+1)(e+1-g)-de-2+2g,
\end{align*}
which we will also refer to as the \textit{expected dimension}\textbf{
}of $\Mor_{e}\left(C,X\right)$. Moreover, as explained in Chapter
2 of \cite{debarre01}, every irreducible component of $\Mor_{e}\left(C,X\right)$
has at least the expected dimension.

\section{\label{sec:Spreading-out-and}Spreading out and the Lang-Weil bound}

In this section, we explain in detail the reduction to finite characteristic.
The arguments are the same as in \cite{BrowningVisheRatCurves}. Again,
fix a smooth projective curve $C$ of genus $g\ge1$, a smooth hypersurface
$X\subset\P^{n}$ of degree $d$, both defined over $\C$.

We follow the strategy in Chapter 3 of \cite{debarre01}. Since $C$
and $X$ are defined by finitely many coefficients in $\C$, let $R$
be the finitely-generated subring of $\C$ (over $\Z$) given by adjoining
to $\Z$ these coefficients. Then, there are projective schemes $\mathcal{C}$
and $\mathcal{X}$ over $R$ such that $C$ and $X$, respectively,
are the base changes to $\C$ of the generic fibers. There is a dense
open subset $U$ of $\Spec R$ over which $\mathcal{C}$ and $\mathcal{X}$
are smooth, and we may shrink (by inverting elements) $\Spec R$ to
assume that $\mathcal{C}\to\Spec R$ and $\mathcal{X}\to\Spec R$
are smooth. By shrinking even more, we may even assume that $R$ is
an integral domain. By Zariski's lemma, the residue field of any closed
point in $\Spec R$ is a finite field, and by inverting even more
elements, we may assume that the characteristics of the residue fields
of closed points are larger than $d$.

Now, there is a quasi-projective relative morphism space $\Mor_{e}\left(\mathcal{C},\mathcal{X}\right)\to\Spec R$
that parametrizes $R$-morphisms $\mathcal{C}\to\mathcal{X}$ (see
Chapter 2 of \cite{debarre01} for more details) such that for any
point $s\in\Spec R$, the base change $\Mor_{e}\left(\mathcal{C},\mathcal{X}\right)\times_{\Spec R}s$
is canonically isomorphic to $\Mor_{e}\left(\mathcal{C}\times_{\Spec R}s,\mathcal{X}\times_{\Spec R}s\right)$. 

We claim that if one is able to show that every fiber of $\Mor_{e}\left(\mathcal{C},\mathcal{X}\right)\to\Spec R$
above a closed point has dimension $m$ and is geometrically irreducible,
then the generic fiber must also have dimension $m$ and be geometrically
irreducible. To see this, note that both of these properties are constructible
properties on the target. Also, $R$ is Jacobson because it is a finitely-generated
$\Z$-algebra, so its closed points are dense in $\Spec R$. Finally,
a constructible subset of $\Spec R$ containing a dense subset of
$\Spec R$ is necessarily all of $\Spec R$.

Hence, we may now assume that $C$ is a smooth projective curve over
$\F_{q}$ of genus $g\ge1$, that $X\subset\P_{\F_{q}}^{n}$ is a
smooth hypersurface of degree $d$, and that the characteristic of
$\F_{q}$ is larger than $d$. 

Recall the Lang-Weil bound from \cite{langweil54}:
\begin{thm}
\label{thm:LW}Let $q$ be a power of a prime. If $X\subset\P_{\F_{q}}^{n}$
is a geometrically irreducible subvariety of degree $d$ and dimension
$r$, then
\[
\left|\#X\left(\F_{q}\right)-q^{r}\right|\le\left(d-1\right)\left(d-2\right)q^{r-1/2}+O_{d,n,r}\left(q^{r-1}\right).
\]
\end{thm}

It then suffices to prove the following theorem, which is a finite
characteristic variant of the problem. 
\begin{thm}
\label{thm:alternative}Let $C$ be a smooth projective curve over
$\F_{q}$ of genus $g$, $X\subset\P_{\F_{q}}^{n}$ a smooth hypersurface
of degree $d$, where $q$ is a power of a prime $p>d$ (and large
enough so that $C\left(\F_{q}\right)\neq\emptyset$), and with $n,d,e,$
and $g$ subject to the same constraints as in Theorem \ref{thm:main}. 

Then,
\[
\lim_{q\to\infty}q^{-\mu}\#\Mor_{e}\left(C,X\right)\left(\F_{q}\right)\le1,
\]
where the limit ranges over larger and larger powers of $p$ and $\mu=(n+1)(e+1-g)-de-2+2g.$
\end{thm}

\begin{proof}
[Proof of Theorem \ref{thm:main} assuming Theorem \ref{thm:alternative}]Suppose
$\Mor_{e}\left(C,X\right)$ has $s\ge1$ geometrically irreducible
components, each of which has dimension $\mu_{i}\ge\mu=(n+1)(e+1-g)-de-2+2g$
by the discussion at the end of Section \ref{sec:A-lower-bound}.
Then, the Lang-Weil bound tells us that $\frac{\#\Mor_{e}\left(C,X\right)\left(\F_{q}\right)}{\sum_{i=1}^{s}q^{\mu_{i}}}\to1,$
so if we can show that $\frac{\#\Mor_{e}\left(C,X\right)\left(\F_{q}\right)}{q^{\mu}}\le1$
as $q\to\infty$, we must necessarily have $\sum_{i=1}^{s}q^{\mu_{i}}\le q^{\mu}$
as $q\to\infty$. But this is only possible if $s=1$ and $\mu_{i}=\mu$. 
\end{proof}

\section{\label{sec:settingupcircle}Setting up the circle method}

Assume we are in the situation of Theorem \ref{thm:alternative}.
In this section, we'll describe how to express point-counts of the
morphism space in terms of a sum of exponential sums. In particular,
we'll explain the ``circle'' over which we are summing, and then
show that for small degree arcs, we can obtain a reasonable bound
that contributes to most of the exponential sum. These small degree
arcs will then by called ``major arcs'' and the remaining large
degree arcs ``minor arcs,'' whose contributions we will bound in
later sections.
\begin{notation}
\label{nota:Fix-a-line}Fix a line bundle $L$ on $C$ of degree $e$.
As described in the introduction, our exponential sums will be taken
over linear functionals of the global sections of a certain line bundle,
namely $L^{\otimes d}$. 

When the choice of line bundle $L$ is clear, we will write $P_{de,C}$
to denote $H^{0}\left(C,L^{\otimes d}\right)$. Similarly, for a closed
subscheme $Z\subset C$, we write $P_{de,Z}$ to denote $H^{0}\left(Z,L^{\otimes d}|_{Z}\right)$
and $P_{de-Z,C}$ to denote $H^{0}\left(C,L^{\otimes d}\left(-Z\right)\right)$. 

Say $\alpha\in P_{de,C}^{\vee}$ factors through a closed subscheme
$Z\subset C$, denoted $\alpha\sim Z$, if it factors through the
restriction map $P_{de,C}^{\vee}\to P_{de,Z}^{\vee}$. If $\alpha$
factors through $Z$ and not through any proper subscheme of $Z$,
we write $\alpha\sim_{\min}Z$. 
\end{notation}

\begin{lem}
\label{lem:degupperbound}Any $\alpha\in P_{de,C}^{\vee}$ factors
through a closed subscheme of degree at most $de/2+1$. 
\begin{proof}
Let $D$ be a divisor and consider the map 
\[
\phi\colon H^{0}\left(C,\mathcal{O}\left(D\right)\right)\to H^{0}\left(C,L^{\otimes d}\left(-D\right)\right)^{\vee}
\]
induced by the cup product:
\[
H^{0}\left(C,\mathcal{O}\left(D\right)\right)\otimes H^{0}\left(C,L^{\otimes d}\left(-D\right)\right)\overset{\cup}{\to}H^{0}\left(C,L^{\otimes d}\right)\overset{\alpha}{\to}\F_{q}.
\]
Then $\alpha$ factors through an effective Cartier divisor linearly
equivalent to $D$ iff $\phi$ is not injective. Indeed, if $s\in\ker\phi$,
then $\alpha$ will factor through the vanishing locus of $s$, since
$\alpha$ vanishes on $H^{0}\left(C,L^{\otimes d}\left(-\divisor(s)\right)\right)$
(note that $\divisor(s)$ is an effective Cartier divisor corresponding
to the vanishing locus of $s$).

If $\deg D=\left\lfloor de/2\right\rfloor +1$, then 
\[
h^{0}\left(C,\mathcal{O}(D)\right)\ge\deg D-g+1>de-\deg D-g+1=h^{0}\left(C,L^{\otimes d}\left(-D\right)\right)
\]
by Riemann-Roch (note that we are using the assumptions on $e$ from
Theorem \ref{thm:main}---in particular, $e>2g-2$), so for dimension
reasons $\phi$ can't be injective.
\end{proof}
\end{lem}

\begin{lem}
Let $Z_{1}$ and $Z_{2}$ be two subschemes of $C$ such that $\alpha\sim Z_{1}$
and $\alpha\sim Z_{2}$. Let $Z$ be the smallest subscheme containing
both, and suppose $\deg Z<de-2g+2$. 

Then $\alpha\sim Z_{1}\cap Z_{2}$. 
\begin{proof}
We have $\alpha|_{P_{de-Z_{1},C}}=\alpha|_{P_{de-Z_{2},C}}=0$. Consider
the span of $P_{de-Z_{1},C}$ and $P_{de-Z_{2},C}$, which both live
in $P_{de-Z_{1}\cap Z_{2},C}$. By Riemann--Roch and the assumptions
on $Z$, we have $\dim P_{de-Z_{i},C}=de-\deg Z_{i}-g+1$ and $\dim P_{de-Z,C}=de-\deg Z-g+1$,
and hence the span has dimension 
\[
\left(de-\deg Z_{1}-g+1\right)+\left(de-\deg Z_{2}-g+1\right)-\left(de-\deg Z-g+1\right)=P_{de-Z_{1}\cap Z_{2},C},
\]
which implies that the span is all of $P_{de-Z_{1}\cap Z_{2},C}$.
Hence, $\alpha|_{P_{de-Z_{1}\cap Z_{2},C}}=0$, so $\alpha\sim Z_{1}\cap Z_{2}$. 
\end{proof}
\end{lem}

\begin{rem}
By the above lemma, note that for any $\alpha$ factoring through
some closed subscheme $Z$ with $\deg Z\le e-2g+1$, there is a unique
minimal subscheme that $\alpha$ factors through.
\end{rem}

We now come to the definition of the exponential sum.
\begin{notation}
\label{nota:expsum}Let 
\[
S(\alpha)\coloneqq\sum_{x\in P_{e,C}^{n+1}}\psi\left(\alpha(f(x))\right),
\]
where $\psi$ is a non-trivial additive character on $\F_{q}$ and
$f$ is the degree $d$ homogeneous equation of the hypersurface $X$. 
\end{notation}

Note that we can think of $f$ as a map from $P_{e,C}^{n+1}$ to $P_{de,C}$
via the cup product. 

There is a natural map $\Mor_{e}\left(C,X\right)\to\Pic^{e}\left(C\right)$
given by pulling back $O_{X}\left(1\right)$ to $C\times T$ along
$C\times T\to X\times T$ for a $T$-point of $\Mor_{e}\left(C,X\right)$.
Denote by $\Mor_{e}\left(C,X,L\right)$ the fiber above a degree $e$
line bundle $L$.

Note that the rational points of $\Mor_{e}\left(C,X,L\right)$ comprise
tuples of the following form:
\[
\left\{ \left(s_{1},\ldots,s_{n+1}\right)\in P_{e,C}^{n+1}:f\left(s_{1},\ldots,s_{n+1}\right)=0,s_{1},\ldots,s_{n+1}\text{ globally generate}\right\} /\G_{m}.
\]
If we pick a basis $v_{1},\ldots,v_{h^{0}\left(C,L\right)}$ of $P_{e,C}$,
each $s_{i}=\sum_{j=1}^{h^{0}\left(C,L\right)}a_{ij}v_{j}$ uniquely,
and the condition $f\left(s_{1},\ldots,s_{n+1}\right)=0$ is equivalent
to polynomial conditions by setting the coefficient of each degree
to zero. If we define 
\[
M_{e}\left(C,X,L\right)\coloneqq\text{these polynomial conditions cutting out a space in }\A^{(n+1)h^{0}\left(C,L\right)},
\]
then some open subset (since we need only an upper bound on the dimension,
it suffices to prove the result ignoring the condition of base-point-freeness)
of $M_{e}\left(C,X,L\right)$ quotiented by $\G_{m}$ is precisely
$\Mor_{e}\left(C,X,L\right)$. Consequently, it suffices to show 
\[
\lim_{q\to\infty}q^{-\widehat{\mu}}\#M_{e}\left(C,X,L\right)\left(\F_{q}\right)\le1,
\]
where the limit ranges over larger and larger powers of $p$ and $\widehat{\mu}=\mu-g+1=(n+1)(e+1-g)-(de+1)+g$,
since $\Pic^{e}\left(C\right)$ is a smooth proper variety of dimension
$g$ by Tag 0B9R of \cite{stacksproj} and \cite{deligne1974conjecture}
(i.e. it has $q^{g}+o\left(q^{g}\right)$ rational points). 

Let us write $CX$ for the affine cone of $X$, i.e. the hypersurface
in $\A^{n+1}$ cut out by $f$. 

Now, note 
\begin{align*}
\sum_{\alpha\in P_{de,C}^{\vee}}S(\alpha) & =\sum_{\alpha\in P_{de,C}^{\vee}}\sum_{x\in P_{e,C}^{n+1}}\psi\left(\alpha\left(f(x)\right)\right)\\
 & =\sum_{x\in P_{e,C}^{n+1}}\sum_{\alpha\in P_{de,C}^{\vee}}\psi\left(\alpha\left(f(x)\right)\right)\\
 & =\sum_{x\in P_{e,C}^{n+1}}\begin{cases}
0 & \text{if }f(x)\ne0\\
\#P_{de,C}^{\vee} & \text{else}
\end{cases}\\
 & =\#P_{de,C}\#M_{e}\left(C,X,L\right)\left(\F_{q}\right).
\end{align*}
By assumption that $e\ge2g-1$ (and applying Riemann--Roch), it follows
that 
\[
\#M_{e}\left(C,X,L\right)\left(\F_{q}\right)=q^{-de+g-1}\sum_{\alpha\in P_{de,C}^{\vee}}S(\alpha).
\]

\begin{lem}
Let $\deg Z\le e-2g+1$. Then, 
\[
\sum_{\alpha\sim Z}S(\alpha)=q^{(n+1)(e+1-g)}\frac{\#CX(Z)}{\#Z^{\dim CX}},
\]
where $\#Z\coloneqq q^{\deg Z}$.
\end{lem}

\begin{proof}
If $\alpha\sim Z$, then 
\begin{align*}
S(\alpha) & =\sum_{x\in P_{e,C}^{n+1}}\psi\left(\alpha\left(f(x)\right)\right)\\
 & =\sum_{f(x)|_{Z}=0}\psi(0)+\sum_{f(x)|_{Z}\ne0}\psi\left(\alpha\left(f(x)\right)\right)\\
 & =\#\left\{ x:f(x)|_{Z}=0\right\} +\sum_{f(x)|_{Z}\ne0}\psi\left(\alpha\left(f(x)\right)\right).
\end{align*}
Also, $P_{de,Z}^{\vee}$ injects into $P_{de,C}^{\vee}$ if $H^{1}\left(C,L^{\otimes d}\left(-Z\right)\right)=0$,
e.g. if $-\deg Z+de>2g-2,$ which is true. So 
\begin{align*}
\sum_{\alpha\sim Z}S(\alpha) & =\#P_{de,Z}\#\left\{ x:f(x)|_{Z}=0\right\} +\sum_{f(x)|_{Z}\ne0}\sum_{\alpha\in P_{de,Z}^{\vee}}\psi\left(\alpha\left(f(x)|_{Z}\right)\right)\\
 & =q^{\deg Z}\#\left\{ x:f(x)|_{Z}=0\right\} .
\end{align*}
We have the exact sequence 
\[
0\to H^{0}\left(C,L\left(-Z\right)\right)\to H^{0}\left(C,L\right)\to H^{0}\left(Z,L|_{Z}\right)\to H^{1}\left(C,L\left(-Z\right)\right),
\]
and since $-\deg Z+e>2g-2$, it follows that $H^{1}\left(C,L\left(-Z\right)\right)=0$,
and so using Riemann-Roch, we have 
\begin{align*}
\sum_{\alpha\sim Z}S(\alpha) & =q^{\deg Z+(n+1)\left(e-\deg Z-g+1\right)}\#CX(Z)\\
 & =q^{(n+1)(e+1-g)}\frac{\#CX(Z)}{\#Z^{\dim CX=n}},
\end{align*}
as desired.
\end{proof}
\begin{notation}
\label{nota:deg}Let us write $\deg\left(\alpha\right)$ to denote
the smallest degree of $Z$ for which $\alpha\sim Z$ and $\left|CX(Z)\right|$
to be $\#CX(Z)/\#Z^{\dim CX}$ (we will use this latter notation only
in this section).
\end{notation}

\begin{cor}
Suppose $Z=Z'+Z''$ as divisors with $Z'$ and $Z''$ disjoint. Then
\[
\sum_{\alpha\sim Z}S(\alpha)=q^{(n+1)(e+1-g)}\left|CX(Z)\right|=q^{-(n+1)(e+1-g)}\sum_{\alpha\sim Z'}S(\alpha)\sum_{\alpha\sim Z''}S(\alpha).
\]

In particular, if $Z=\Spec\left(\mathcal{O}_{x}/\mf m_{x}^{r}\right)$,
then 
\[
\sum_{\alpha\sim_{\min}Z}S(\alpha)=q^{(n+1)(e+1-g)}\left(\left|CX\left(\Spec\left(\mathcal{O}_{x}/\mf m_{x}^{r}\right)\right)\right|-\left|CX\left(\Spec\left(\mathcal{O}_{x}/\mf m_{x}^{r-1}\right)\right)\right|\right),
\]
and, more generally, by the principle of inclusion-exclusion, if $Z=\sum r_{i}x_{i}$,
then 
\begin{align*}
\sum_{\alpha\sim_{\min}Z}S(\alpha) & =q^{(n+1)(e+1-g)}\prod_{i}\left(\sum_{\alpha\sim_{\min}r_{i}x_{i}}S(\alpha)\right)\\
 & =q^{(n+1)(e+1-g)}\prod_{i}\left(\left|CX\left(r_{i}x_{i}\right)\right|-\left|CX\left(\left(r_{i}-1\right)x_{i}\right)\right|\right).
\end{align*}
\end{cor}

\begin{lem}
Using the previous notation, 
\[
\left|CX\left(\Spec\left(\mathcal{O}_{x}/\mf m_{x}^{r}\right)\right)\right|=\frac{1-\#\kappa(x)^{\left\lceil r/d\right\rceil \left(d-n-1\right)}}{1-\#\kappa(x)^{d-n-1}}\left|\left(CX\backslash0\right)\left(\kappa(x)\right)\right|+\#\kappa(x)^{-\left(n+1\right)\left\lceil r/d\right\rceil +r}.
\]
 
\begin{proof}
Let us write $R=\mathcal{O}_{C,x}$ and $\pi$ a uniformizer. Also,
let $\kappa(x)=R/\pi$ and $CX_{i}$ be the points $\left(\pi^{i}x_{0},\ldots,\pi^{i}x_{n}\right)\in CX\left(R/\pi^{r}\right)$
such that $x_{j}\not\in(\pi)$ for some $j$. 

Then, 
\[
CX\left(R/\pi^{r}\right)=\bigsqcup_{i=0}^{r}CX_{i}.
\]

To compute $\#CX\left(R/\pi^{r}\right)$, we compute $\#CX_{i}$ over
the three ranges $0\le i\le\left\lceil r/d\right\rceil -1$, $\left\lceil r/d\right\rceil \le i\le r-1$,
and $i=r$. 

For $0\le i\le\left\lceil r/d\right\rceil -1$, we have that $\left(\pi^{i}x_{0},\ldots,\pi^{i}x_{n}\right)\in CX\left(R/\pi^{r}\right)$
is equivalent to $f\left(\pi^{i}x_{0},\ldots,\pi^{i}x_{n}\right)\in\left(\pi^{r}\right)$,
i.e. $\left(x_{0},\ldots,x_{n}\right)\in CX\left(R/\pi^{r-di}\right)$.
There are $\#\left(CX\backslash0\right)\left(R/\pi^{r-di}\right)$
ways to choose a non-zero element of $CX\left(R/\pi^{r-di}\right)$,
as well as $\#\kappa(x)$ choices for each of the coefficients of
$\pi^{r-di},\ldots,\pi^{r-i-1}$. 

For $\left\lceil r/d\right\rceil \le i\le r-1$, we have that $\left(\pi^{i}x_{0},\ldots,\pi^{i}x_{n}\right)\in CX\left(R/\pi^{r}\right)$
is equivalent to $f\left(\pi^{i}x_{0},\ldots,\pi^{i}x_{n}\right)\in\left(\pi^{r}\right)$,
which is automatic since $di\ge r$. Hence, there are $\#\kappa(x)$
choices for each for the coefficients of the powers $\pi$ through
$\pi^{r-i-1}$, and there are $\#\kappa(x)^{n+1}-1$ choices for the
constant term (not all $n+1$ terms can be zero).

Finally, for $i=r$, this is simply the zero vector.

Hence, we have 
\begin{align*}
 & \#CX\left(R/\pi^{r}\right)\\
 & =\sum_{i=0}^{\left\lceil r/d\right\rceil -1}\#\kappa(x)^{(n+1)(d-1)i}\#\left(CX\backslash0\right)\left(R/\pi^{r-di}\right)+\sum_{i=\left\lceil r/d\right\rceil }^{r-1}\#\kappa(x)^{(n+1)(r-i-1)}\left(\#\kappa(x)^{n+1}-1\right)+1\\
 & =\sum_{i=0}^{\left\lceil r/d\right\rceil -1}\#\kappa(x)^{(n+1)(d-1)i}\#\left(CX\backslash0\right)\left(R/\pi^{r-di}\right)+\#\kappa(x)^{(n+1)\left(r-\left\lceil r/d\right\rceil \right)}.
\end{align*}

Since $CX\backslash0$ is smooth (by assumption on $f$), Hensel's
lemma tells us that $\left|\left(CX\backslash0\right)\left(R/\pi^{r-di}\right)\right|=\left|\left(CX\backslash0\right)\left(\kappa(x)\right)\right|$.
Then, we have 
\begin{align*}
\left|CX\left(R/\pi^{r}\right)\right| & =\sum_{i=0}^{\left\lceil r/d\right\rceil -1}\#\kappa(x)^{(n+1)(d-1)i-din}\frac{\#\left(CX\backslash0\right)\left(R/\pi^{r-di}\right)}{\#\kappa(x)^{(r-di)n}}+\#\kappa(x)^{(n+1)\left(r-\left\lceil r/d\right\rceil \right)-rn}\\
 & =\sum_{i=0}^{\left\lceil r/d\right\rceil -1}\#\kappa(x)^{i(d-n-1)}\left|\left(CX\backslash0\right)\left(\kappa(x)\right)\right|+\#\kappa(x)^{-\left(n+1\right)\left\lceil r/d\right\rceil +r}\\
 & =\frac{1-\#\kappa(x)^{\left\lceil r/d\right\rceil \left(d-n-1\right)}}{1-\#\kappa(x)^{d-n-1}}\left|\left(CX\backslash0\right)\left(\kappa(x)\right)\right|+\#\kappa(x)^{-\left(n+1\right)\left\lceil r/d\right\rceil +r}.
\end{align*}
\end{proof}
\end{lem}

\begin{cor}
Let $r\ge2$. Then, 
\begin{align*}
 & \left|CX\left(\Spec\left(\mathcal{O}_{x}/\mf m_{x}^{r}\right)\right)\right|-\left|CX\left(\Spec\left(\mathcal{O}_{x}/\mf m_{x}^{r-1}\right)\right)\right|\\
 & =\begin{cases}
\left|\left(CX\backslash0\right)\left(\kappa(x)\right)\right|\#\kappa(x)^{\left(\left\lceil r/d\right\rceil -1\right)\left(d-n-1\right)}+\#\kappa(x)^{-\left(n+1\right)\left\lceil r/d\right\rceil +r} & \text{if }r\equiv1\bmod d,\\
\qquad-\#\kappa(x)^{-\left(n+1\right)\left\lceil (r-1)/d\right\rceil +r-1}\\
\#\kappa(x)^{-\left(n+1\right)\left\lceil r/d\right\rceil +r}-\#\kappa(x)^{-\left(n+1\right)\left\lceil (r-1)/d\right\rceil +r-1} & \text{else}.
\end{cases}
\end{align*}
\end{cor}

By Deligne's resolution of the Weil conjectures in \cite{deligne1974conjecture},
we have 
\[
\left|\#X\left(\F_{q}\right)-\frac{q^{n}-1}{q-1}\right|\le Bq^{(n-1)/2}
\]
for some $B$ independent of $q$; if $d=1$, then $B=0$. It follows
that 
\[
\left|\#CX\left(\kappa\left(x\right)\right)-\#\kappa(x)^{n}\right|\le B\#\kappa(x)^{(n-1)/2}\left(\#\kappa(x)-1\right),
\]
so 
\[
\left|\left|CX\left(\kappa\left(x\right)\right)\right|-1\right|\le B'\#\kappa(x)^{(-n+1)/2}
\]
for some $B'$ independent of $q$; if $d=1$, then $B'=0$ as well.

In particular, 
\[
\left|\left|CX\left(\kappa\left(x\right)\right)\right|-1\right|q^{\deg x}\le q^{-1}
\]
under the assumptions of Theorem (\ref{thm:main}). 

For $r\ge2$, the previous corollary similarly tells us that 
\[
\left|\left|CX\left(rx\right)\right|-\left|CX\left(\left(r-1\right)x\right)\right|\right|q^{r\deg x}\le q^{-1}
\]
as well.

Then, we have the bound

\begin{align*}
\frac{\sum_{\alpha:\deg(\alpha)\le e-2g+1}S(\alpha)}{q^{(n+1)(e+1-g)}}\le & 1+\frac{\sum_{D=1}^{e-2g+1}\sum_{\deg(\alpha)=D}S(\alpha)}{q^{(n+1)(e+1-g)}}\\
= & 1+\sum_{D=1}^{e-2g+1}\sum_{\deg(\alpha)=D,\alpha=\sum_{i}r_{i}x_{i}}\prod_{i}\left|\left|CX\left(r_{i}x_{i}\right)\right|-\left|CX\left(\left(r_{i}-1\right)x_{i}\right)\right|\right|\\
\le & 1+\sum_{D=1}^{e-2g+1}\sum_{\deg(\alpha)=D,\alpha=\sum_{i}r_{i}x_{i}}\prod_{i}q^{-r_{i}\deg x_{i}-1}\\
\le & 1+\sum_{D=1}^{e-2g+1}\sum_{\deg(\alpha)=D,\alpha=\sum_{i}r_{i}x_{i}}q^{-D-1}.
\end{align*}
By the Lang--Weil bounds (Theorem (\ref{thm:LW})), we have $\#C\left(\F_{q^{r}}\right)=O\left(q^{r}\right)$. 

Hence, 
\begin{align*}
\frac{\sum_{\alpha:\deg(\alpha)\le e-2g+1}S(\alpha)}{q^{(n+1)(e+1-g)}}-1 & =O\left(\sum_{D=1}^{e-2g+1}q^{-1}\right).
\end{align*}

Then, as $q\to\infty$, this expression goes to 0, so we have 
\[
q^{-(n+1)(e+1-g)}\sum_{\alpha:\deg(\alpha)\le e-2g+1}S(\alpha)\to1.
\]

We now define the major and minor arcs of our sum, based on this computation.
\begin{notation}
\label{nota:majorminor}Say $\alpha\in P_{de,C}^{\vee}$ is a \textit{major
arc} if $\deg(\alpha)\le e-2g+1$. Otherwise, say $\alpha$ is a\textbf{
}\textit{minor arc}. 
\end{notation}

To complete the proof of our main theorem, it then suffices to check
the sum over minor arcs goes to 0, namely
\[
q^{-(n+1)(e+1-g)}\sum_{\alpha:\deg(\alpha)>e-2g+1}S(\alpha)\to0.
\]

\section{Weyl differencing}

Let us use the same notation as in the previous section, i.e. we have
fixed a line bundle $L$ on a smooth projective curve $C$ over $\F_{q}$
and defined the arcs as linear functionals on $P_{de,C}=H^{0}\left(C,L^{\otimes d}\right)$.
See Notations \ref{nota:Fix-a-line}, \ref{nota:expsum}, \ref{nota:deg},
and \ref{nota:majorminor}.

In this step, we establish a Weyl differencing bound that compares
$S(\alpha)$ to $N(\alpha)$, a standard quantity from analytic number
theory defined from multilinear forms associated to the equation of
the hypersurface. The precise statement of this comparison is given
by Proposition \ref{prop:weyldifference}. The strategy here is identical
to that of \cite{schmidt85,lee11}, although we have adapted their
lemmas to our more geometric setting. 
\begin{rem}
The intuition for passing from $S(\alpha)$ to $N(\alpha)$ is that
an exponential sum of the form $\sum_{x\in P_{e,C}}\psi\left(f(x)\right)$
is much easier to work with when $f$ is linear (compared to when
$f$ is a high degree polynomial). Weyl differencing allows one to
pass to the linear case beginning with the observation that 
\begin{align*}
|S(\alpha)|^{2} & =\sum_{x\in P_{e,C}}\sum_{y\in P_{e,C}}\psi\left(\alpha(f(x)-f(y)\right)\\
 & =\sum_{x\in P_{e,C}}\sum_{y\in P_{e,C}}\psi\left(\alpha\left(f(x+y)-f(y)\right)\right)\\
 & \le\sum_{x\in P_{e,C}}\left|\sum_{y\in P_{e,C}}\psi\left(\alpha\left(f(x+y)-f(y)\right)\right)\right|
\end{align*}
by the change of variables $x\mapsto x+y$. Notice that $f(x+y)-f(y)$
(the ``differencing'' operation) is a polynomial in $y$ of degree
one lower than that of $f$. By a repeated application of this differencing
operation and Cauchy--Schwarz, we are able to reduce from bounding
an exponential sum of a high degree polynomial (namely $f$) to an
exponential sum of a linear function. 

The cost of this running this argument is reflected in the fact that
we require exponentially many variables compared to the degree of
the polynomial in the statement of Theorem \ref{thm:main}.
\end{rem}

With $f$ the homogeneous equation cutting out our hypersurface $X$,
we define multilinear forms
\[
f_{d}\left(x^{(1)},\dots,x^{(d)}\right)\coloneqq\sum_{\ep_{1}=0}^{1}\cdots\sum_{\ep_{d}=0}^{1}(-1)^{\ep_{1}+\cdots+\ep_{d}}f\left(\ep_{1}x^{(1)}+\cdots+\ep_{d}x^{(d)}\right),
\]
where each $x^{(i)}$ is a tuple $\left(x_{1}^{(i)},\ldots,x_{n+1}^{(i)}\right)\in P_{e,C}^{n+1}$.
Write 
\[
f=\sum_{j_{1},\ldots,j_{d}=1}^{n+1}a_{j_{1},\ldots,j_{d}}x_{j_{1}}\cdots x_{j_{d}},
\]
where $a_{j_{1},\ldots,j_{d}}$ is symmetric in its entries (recall
$f$ is a degree $d$ polynomial). For instance, for $d=2$ and $n=1$,
this would look like 
\[
f=a_{1,1}x_{1}^{2}+a_{1,2}x_{1}x_{2}+a_{2,1}x_{2}x_{1}+a_{2,2}x_{2}^{2}=a_{1,1}x_{1}^{2}+2a_{1,2}x_{1}x_{2}+a_{2,2}x_{2}^{2}.
\]
 
\begin{notation}
\label{nota:psij}Let 
\[
\Psi_{j}\left(x^{(1)},\ldots,x^{(d-1)}\right)=d!\sum_{j_{1},\ldots,j_{d-1}=1}^{n+1}a_{j_{1},\ldots,j_{d-1},j}x_{j_{1}}^{(1)}\cdots x_{j_{d-1}}^{(d-1)},
\]
where multiplication should be interpreted via the cup product (the
presence of $d!$ in this expression is why we assume $p>d$). 
\end{notation}

Then, by plugging in the definition and cancelling terms, it is evident
that 

\[
f_{d}\left(x^{(1)},\dots,x^{(d-1)},x\right)=\Psi_{1}\left(x^{(1)},\ldots,x^{(d-1)}\right)x_{1}+\cdots+\Psi_{n+1}\left(x^{(1)},\ldots,x^{(d-1)}\right)x_{n+1},
\]
where $x=\left(x_{1},\ldots,x_{n+1}\right)\in P_{e,C}^{n+1}$. So,
we get
\begin{align*}
 & \sum_{x\in P_{e,C}^{n+1}}\psi\left(\alpha\left(f_{d}\left(x^{(1)},\dots,x^{(d-1)},x\right)\right)\right)\\
= & \sum_{x_{1}\in P_{e,C}}\cdots\sum_{x_{n+1}\in P_{e,C}}\psi\left(\alpha\left(\Psi_{1}\left(x^{(1)},\ldots,x^{(d-1)}\right)x_{1}+\cdots+\Psi_{n+1}\left(x^{(1)},\ldots,x^{(d-1)}\right)x_{n+1}\right)\right).
\end{align*}
Then, the same proof as that of Lemma 11.1 from \cite{schmidt85}
gives the following:
\begin{align*}
 & \left|S(\alpha)\right|^{2^{d-1}}\\
 & \le\left(\#P_{e,C}^{n+1}\right)^{2^{d-1}-d}\sum_{x^{(1)}\in P_{e,C}^{n+1}}\cdots\sum_{x^{(d-1)}\in P_{e,C}^{n+1}}\left|\sum_{x\in P_{e,C}^{n+1}}\psi\left(\alpha\left(f_{d}\left(x^{(1)},\dots,x^{(d-1)},x\right)\right)\right)\right|\\
 & =\left(\#P_{e,C}^{n+1}\right)^{2^{d-1}-d}\sum_{x^{(1)}\in P_{e,C}^{n+1}}\cdots\sum_{x^{(d-1)}\in P_{e,C}^{n+1}}\prod_{j=1}^{n+1}\left|\sum_{x\in P_{e,C}}\psi\left(\alpha\left(\Psi_{j}\left(x^{(1)},\ldots,x^{(d-1)}\right)x\right)\right)\right|.
\end{align*}
Note for any linear functional $\gamma$ on $P_{e,C}$, we have 
\[
\sum_{x\in P_{e,C}}\psi\left(\gamma(x)\right)=\begin{cases}
\#P_{e,C} & \gamma=0\\
0 & \text{otherwise}
\end{cases}.
\]

\begin{notation}
\label{nota:N(alpha)}Let 
\begin{align*}
N(\alpha) & \coloneqq\#\left\{ \left(x^{(1)},\ldots,x^{(d-1)}\right)\in\left(P_{e,C}^{n+1}\right)^{d-1}\colon\alpha\left(\Psi_{j}\left(x^{(1)},\ldots,x^{(d-1)}\right)x\right)=0\text{ for }x\in P_{e,C}\text{ and all }j\right\} .
\end{align*}
\end{notation}

As a consequence of the discussion above, we have 
\begin{prop}
\label{prop:weyldifference}
\[
\left|S(\alpha)\right|^{2^{d-1}}\le\left(\#P_{e,C}^{n+1}\right)^{2^{d-1}-d+1}N(\alpha).
\]
\end{prop}

\section{Why is the circle linear functionals on $P_{de,C}$?}

Let us keep the notation from the previous sections (see Notations
\ref{nota:Fix-a-line}, \ref{nota:expsum}, \ref{nota:deg}, and \ref{nota:majorminor}).
We now explain the connection between $P_{de,C}^{\vee}$ and $\mathbb{T}=\left\{ \alpha\in K_{\infty}:|\alpha|<1\right\} $,
the latter of which is used in Browning and Vishe's formation of the
exponential sum. 
\begin{notation}
\label{nota:choiceofD}Recall $K_{\infty}$ is the fraction field
of the completion of $\mathcal{O}_{C,\infty}$, where $\infty$ is
simply some fixed $\F_{q}$-point (by assumption on $q$, such a point
exists). Let $D$ be a divisor such that $\mathcal{O}(D)\cong L$.
For convenience, and since $D$ and $K_{C}$ are only defined up to
linear equivalence, we may assume $v_{\infty}\left(D\right)=v_{\infty}\left(K_{C}\right)=0$.
\end{notation}

In this section, we explain, among other things, how $\alpha$ can
be lifted to an element of $K_{\infty}$.

By Serre duality (see the corollary to Theorem 5 of \cite{tate68}),
there is a residue pairing 
\[
H^{0}\left(C,\mathcal{O}\left(dD\right)\right)\times H^{1}\left(C,\mathcal{O}\left(K_{C}-dD\right)\right)\to\F_{q},
\]
where we identify $H^{1}\left(C,\mathcal{O}\left(K_{C}-dD\right)\right)$
with the group of repartitions (see Section 4 of \cite{tate68})
\[
\frac{R}{R\left(K_{C}-dD\right)+K(C)},
\]
where $R$ is the adele ring of $K(C)$. Here, an element of $R\left(K_{C}-dD\right)$
is a tuple $\left(a_{p}\right)_{p\text{ closed point}}$ such that
$v_{p}\left(a_{p}\right)\ge-v_{p}\left(K_{C}-dD\right)$. 

More precisely, the pairing is given by 
\[
\wangle{\omega,r}=\sum_{p\text{ closed point}}\Res_{p}\left(\omega r_{p}\right).
\]
 Let $R_{\not\ne\infty}$ be the adeles without the contribution at
$\infty$.

Consider the map 
\[
K(C)\to\frac{R_{\not\ne\infty}}{R_{\not\ne\infty}\left(K_{C}-dD\right)}
\]
sending $f$ to $\left[\left(f_{p}\right)_{p\ne\infty}\right].$ 

By strong approximation (Lemma 9.1.9 of \cite{qingliubook}) this
is surjective. The kernel $K(C)\left(K_{C}-dD\right)|_{\ne\infty}$
is given by elements $f\in K(C)$ such that $v_{p}(f)\ge v_{p}\left(-K_{C}+dD\right)$
for all closed points $p\ne\infty$. 

Consider, now, the following abstract setup: 

Let $B_{i}\subset A_{i}$ and $C\subset A_{1}\times A_{2}$ be abelian
groups. Then, we have an exact sequence 
\begin{equation}
0\to\frac{A_{2}}{B_{2}+\im\left(\ker\left(C\to A_{1}/B_{1}\right)\to A_{2}\right)}\to\frac{A_{1}\times A_{2}}{B_{1}\times B_{2}+C}\to\frac{A_{1}}{B_{1}+\im(C\to A_{1})}\to0.\label{eq:exactsequence}
\end{equation}
Note that $\times$ is used to refer to an abstract product, whereas
$+$ used for addition with the obvious embedding. For example, $B_{1}\times B_{2}$
and $C$ are both embedded inside $A_{1}\times A_{2}$, and $B_{1}\times B_{2}+C$
is the sum taken inside $A_{1}\times A_{2}.$ 

To see this, first observe the simpler exact sequence 
\[
0\to\frac{\im(P\to M)\times N}{P}\to\frac{M\times N}{P}\to\frac{M}{\im(P\to M)}\to0
\]
for abelian groups $M,N$, and $P$ satisfying $P\subset M\times N$.
Replacing the term $\frac{\im(P\to M)\times N}{P}$ with $\frac{N}{\im\left(\ker(P\to M)\to N\right)}$
via the isomorphism 
\[
\frac{N}{\im\left(\ker(P\to M)\to N\right)}\to\frac{\im(P\to M)\times N}{P}
\]
 given by $[n]\mapsto[0\times n],$ we obtain the exact sequence 
\[
0\to\frac{N}{\im\left(\ker(P\to M)\to N\right)}\to\frac{M\times N}{P}\to\frac{M}{\im(P\to M)}\to0.
\]
Now, set $M=A_{1}/B_{1},N=A_{2}/B_{2},$ and $P=\im\left(C\to\left(A_{1}\times A_{2}\right)/\left(B_{1}\times B_{2}\right)\right)$. 

Setting $A_{1}=R_{\not\ne\infty},A_{2}=K_{\infty},B_{1}=R_{\not\ne\infty}\left(K_{C}-dD\right),B_{2}=K_{\infty}\left(K_{C}-dD\right),$
and $C=K(C)$ in the exact sequence (\ref{eq:exactsequence}) above,
we obtain

\begin{align*}
\frac{K_{\infty}}{K_{\infty}\left(K_{C}-dD\right)+K(C)\left(K_{C}-dD\right)|_{\ne\infty}} & \cong\frac{R}{R\left(K_{C}-dD\right)+K(C)},\\
 & \cong H^{1}\left(C,\mathcal{O}\left(K_{C}-dD\right)\right)\\
 & \cong H^{0}\left(C,\mathcal{O}\left(dD\right)\right)^{\vee},
\end{align*}
i.e. 
\begin{equation}
P_{de,C}^{\vee}\cong\frac{K_{\infty}}{K_{\infty}\left(K_{C}-dD\right)+K(C)\left(K_{C}-dD\right)|_{\ne\infty}}\label{eq:lifttoKinfty}
\end{equation}
which shows that we can lift any element of $P_{de,C}^{\vee}$ to
$K_{\infty}$. 

Of course, by the same reasoning, we have 
\begin{equation}
P_{e,C}^{\vee}\cong\frac{K_{\infty}}{K_{\infty}\left(K_{C}-D\right)+K(C)\left(K_{C}-D\right)|_{\ne\infty}}\label{eq:lifttoKinftyII}
\end{equation}
for any divisor $D$, from which it follows that any element $\beta\in P_{e,C}^{\vee}$
that vanishes on every element of $P_{e,C}$ can be written as $\beta'+\beta''$,
where $\beta'\in\mathcal{O}\left(K_{C}-D\right)|_{C-\infty}$ and
$\beta''\in\widehat{\mathcal{O}_{C,\infty}}$. Indeed, note that $K_{\infty}\left(K_{C}-D\right)$
is the same as $\widehat{\mathcal{O}_{C,\infty}}$ because of our
assumption on $D$ (see Notation \ref{nota:choiceofD}).
\begin{notation}
\label{nota:abuseofnotation}Note that $\mathcal{O}\left(K_{C}-D\right)|_{C-\infty}$
and $K(C)\left(K_{C}-D\right)|_{\ne\infty}$ are the same, although
by abuse of notation we will occasionally identify $\mathcal{O}\left(K_{C}-D\right)|_{C-\infty}$
with its global sections (in any case this is not an important distinction
because $C-\infty$ is affine).
\end{notation}

\begin{rem}
Alternatively, we can think about this from the perspective of the
Beauville--Laszlo theorem (see Remark 2 of \cite{BL95}). 

More generally, suppose $\beta\in K_{\infty}$ is the lift of an element
in $H^{1}\left(C,\mathcal{O}\left(K_{C}-D\right)\right)$ (which,
by Serre duality, is the same as a linear functional in $P_{e,C}^{\vee}$).
Using the Beauville--Laszlo theorem, consider the vector bundle $V$
on $C$ defined by $\mathcal{O}\left(K_{C}-D\right)|_{C-\infty}\oplus\mathcal{O}|_{C-\infty}$
on $C-\infty$, $\widehat{\mathcal{O}_{C,\infty}}^{2}$ on $\widehat{\mathcal{O}_{C,\infty}}$,
and the transition matrix 
\[
\begin{bmatrix}1 & -\beta\\
0 & 1
\end{bmatrix}.
\]
This fits naturally into a short exact sequence $0\to\mathcal{O}\left(K_{C}-D\right)\to V\to\mathcal{O}\to0$,
which defines an extension of $\mathcal{O}$ by $\mathcal{O}\left(K_{C}-D\right)$
and hence an element of $H^{1}\left(C,\mathcal{O}\left(K_{C}-D\right)\right)$. 
\begin{claim}
This extension corresponds to $\beta$. In particular, if the linear
functional is 0, then $\beta$ can be written as $\beta'+\beta''$,
where $\beta'\in\mathcal{O}\left(K_{C}-D\right)|_{C-\infty}$ and
$\beta''\in\widehat{\mathcal{O}_{C,\infty}}$. 
\end{claim}

\begin{proof}
The extension is defined on $C-\infty$ by the obvious short exact
sequence 
\[
0\to\mathcal{O}\left(K_{C}-D\right)|_{C-\infty}\to\mathcal{O}\left(K_{C}-D\right)|_{C-\infty}\oplus\mathcal{O}|_{C-\infty}\to\mathcal{O}|_{C-\infty}\to0
\]
and similarly on $\widehat{\mathcal{O}_{C,\infty}}$ by 
\[
0\to\widehat{\mathcal{O}_{C,\infty}}\to\widehat{\mathcal{O}_{C,\infty}}\oplus\widehat{\mathcal{O}_{C,\infty}}\to\widehat{\mathcal{O}_{C,\infty}}\to0
\]
(inclusion into the first coordinate, then projection onto the second
coordinate). 

Consider, now, the short exact sequence 
\[
0\to\mathcal{O}\left(K_{C}-D\right)\to\mathcal{K}(C)\to S\to0,
\]
where $\mathcal{K}(C)=\underline{K(C)}$ is the sheaf of meromorphic
functions on $C$, and $S$ is the cokernel. We define a morphism
between these short exact sequences on $C-\infty$ and $\widehat{\mathcal{O}_{C,\infty}}$
and glue them. 

First, on $C-\infty$, we have the following diagram:

\[\begin{tikzcd} 	0 & {\mathcal{O}\left(K_{C}-D\right)|_{C-\infty}} & {\mathcal{O}\left(K_{C}-D\right)|_{C-\infty}\oplus\mathcal{O}|_{C-\infty}} & {\mathcal{O}|_{C-\infty}} & 0 \\ 	0 & {\mathcal{O}\left(K_{C}-D\right)|_{C-\infty}} & {\mathcal{K}(C)} & {S|_{C-\infty}} & 0 	\arrow[from=1-1, to=1-2] 	\arrow[from=1-2, to=1-3] 	\arrow[from=1-3, to=1-4] 	\arrow[from=1-4, to=1-5] 	\arrow[from=2-1, to=2-2] 	\arrow[from=2-2, to=2-3] 	\arrow[from=2-3, to=2-4] 	\arrow[from=2-4, to=2-5] 	\arrow["{=}", from=1-2, to=2-2] 	\arrow["{(x,y)\mapsto x}", from=1-3, to=2-3] 	\arrow["{y\mapsto 0}", from=1-4, to=2-4] \end{tikzcd}\]

On $\widehat{\mathcal{O}_{C,\infty}},$ we similarly have: 

\[\begin{tikzcd} 	0 & {\widehat{\mathcal{O}_{C,\infty}}} & {\widehat{\mathcal{O}_{C,\infty}}\oplus\widehat{\mathcal{O}_{C,\infty}}} & {\widehat{\mathcal{O}_{C,\infty}}} & 0 \\ 	0 & {\widehat{\mathcal{O}_{C,\infty}}} & {K_\infty} & {S|_{\widehat{\mathcal{O}_{C,\infty}}}} & 0 	\arrow[from=1-1, to=1-2] 	\arrow[from=1-2, to=1-3] 	\arrow[from=1-3, to=1-4] 	\arrow[from=1-4, to=1-5] 	\arrow[from=2-1, to=2-2] 	\arrow[from=2-2, to=2-3] 	\arrow[from=2-3, to=2-4] 	\arrow[from=2-4, to=2-5] 	\arrow["{=}", from=1-2, to=2-2] 	\arrow["{(x,y)\mapsto x+\beta y}", from=1-3, to=2-3] 	\arrow["{y\mapsto [\beta y]}", from=1-4, to=2-4] \end{tikzcd}\]

Note that these are compatible on the overlap where $(x,y)\in\mathcal{O}\left(K_{C}-D\right)|_{C-\infty}\oplus\mathcal{O}|_{C-\infty}$
satisfy $\left(x-\beta y,y\right)\in\widehat{\mathcal{O}_{C,\infty}}^{2}$,
since in the second column, for instance, we have $(x,y)\mapsto x$
for $C-\infty$, as well as $\left(x-\beta y,y\right)\mapsto x-\beta y+\beta y=x$
for $\widehat{\mathcal{O}_{C,\infty}}.$ So, together, these glue
to give us a morphism of short exact sequences. The induced bottom
long exact sequence identifies $S/\left(S\left(K_{C}-D\right)+K(C)\right)\cong R/\left(R\left(K_{C}-D\right)+K(C)\right)$
with $H^{1}\left(C,\mathcal{O}\left(K_{C}-D\right)\right)$ (see Section
4 of \cite{tate68}), where $R$ is the adele ring of $K(C)$. 

Also, the image of 1 in $H^{0}\left(C,\mathcal{O}\right)\to H^{1}\left(C,\mathcal{O}\left(K_{C}-D\right)\right)$
from the induced top long exact sequence is precisely the cohomology
class of the extension defined by $\beta$. Hence, the image of 1
in the vertical map $H^{0}\left(C,\mathcal{O}\right)\to H^{0}\left(C,S\right)$
gives the corresponding lift of this cohomology class:

\[\begin{tikzcd} 	{} & {H^0(C,\mathcal{O})} & {H^{1}\left(C,\mathcal{O}\left(K_{C}-D\right)\right)} & {} \\ 	{} & {H^0(C,S)} & {H^{1}\left(C,\mathcal{O}\left(K_{C}-D\right)\right)} & {} 	\arrow[from=1-1, to=1-2] 	\arrow[from=1-2, to=1-3] 	\arrow[from=1-2, to=2-2] 	\arrow[from=1-3, to=1-4] 	\arrow["\cong"', from=1-3, to=2-3] 	\arrow[from=2-1, to=2-2] 	\arrow[from=2-2, to=2-3] 	\arrow[from=2-3, to=2-4] \end{tikzcd}\]

But, over $C-\infty$, 1 maps to 0. Over $\widehat{\mathcal{O}_{C,\infty}}$,
1 maps to $\beta\in K_{\infty}$. Then, using the identification $K_{\infty}/$$\left(K_{\infty}\left(K_{C}-D\right)+K(C)\left(K_{C}-D\right)|_{\ne\infty}\right)$
with $R/\left(R\left(K_{C}-D\right)+K(C)\right)$ as we did earlier,
we see that the class of $\beta$ in $H^{1}\left(C,\mathcal{O}\left(K_{C}-D\right)\right)$
is the same as the class of the extension.

For the second claim where we further assume the cohomology class
of the extension is trivial, this means the extension splits. In other
words, there are invertible matrices $\begin{bmatrix}1 & \beta''\\
0 & 1
\end{bmatrix}$ and $\begin{bmatrix}1 & \beta'\\
0 & 1
\end{bmatrix}$ with $\beta''\in\widehat{\mathcal{O}_{C,\infty}}$ (see the discussion
preceding Notation \ref{nota:abuseofnotation}) and $\beta'\in\mathcal{O}\left(K_{C}-D\right)|_{C-\infty}$
(see Notation \ref{nota:abuseofnotation}) such that the composition
$\begin{bmatrix}1 & \beta''\\
0 & 1
\end{bmatrix}\begin{bmatrix}1 & -\beta\\
0 & 1
\end{bmatrix}\begin{bmatrix}1 & \beta'\\
0 & 1
\end{bmatrix}$ is the identity. So the result follows.
\end{proof}
\end{rem}

\section{The geometry of numbers via vector bundles}

We continue to use the notation from the previous sections (see Notations
\ref{nota:Fix-a-line}, \ref{nota:expsum}, \ref{nota:deg}, \ref{nota:majorminor},
\ref{nota:psij}, \ref{nota:N(alpha)}, \ref{nota:choiceofD}, and
\ref{nota:abuseofnotation}). In this step, we relate $N(\alpha)$
to a slight modification denoted by $N_{s}(\alpha)$, which in \cite{BrowningVisheRatCurves}
is called Davenport's shrinking lemma. The precise statement in our
setting is given by Proposition \ref{prop:shrinking}.

To prove Davenport's shrinking lemma in their setting, Browning and
Vishe use facts about the geometry of numbers in the function field
setting from \cite{lee11}, which we reinterpret geometrically. Actually,
Browning and Sawin prove a better version of the shrinking lemma (Lemma
6.4 of \cite{cohomology_circle}) also using facts about successive
minima from \cite{lee11}, which we adapt to work in our geometric
setting. In particular, instead of working with lattices and successive
minima, we work with vector bundles over $C$ and slopes arising from
the Harder--Narasimhan filtration. 

\begin{notation}
Define 
\begin{align*}
N_{s,\ell}(\alpha) & \coloneqq\#\left\{ \left(x^{(1)},\ldots,x^{(\ell)}\right)\in\left(P_{e-s,C}^{n+1}\right)^{\ell},\left(x^{(\ell+1)},\ldots,x^{(d-1)}\right)\in\left(P_{e,C}^{n+1}\right)^{d-1-\ell}\colon\right.\\
 & \quad\quad\left.\alpha\left(\Psi_{j}\left(x^{(1)},\ldots,x^{(d-1)}\right)x\right)=0\text{ for }x\in P_{e+\ell s,C}\text{ and all }j\right\} .
\end{align*}
Here, $P_{e+\ell s,C}=H^{0}\left(C,L\left(\ell s\infty\right)\right)$
for any integer $\ell$ between $0$ and $d-1$, inclusive. In particular,
let 
\[
N_{s}(\alpha)=N_{s,d-1}(\alpha).
\]
\end{notation}

To obtain the inequality involving $N(\alpha)=N_{s,0}(\alpha)$ and
$N_{s}(\alpha),$ we successively compare $N_{s,\ell-1}(\alpha)$
and $N_{s,\ell}(\alpha)$, which we do through the global sections
of a specific vector bundle. 

First, we pick a lift of $\alpha$ to $K_{\infty}$ using the identification
(\ref{eq:lifttoKinfty}) from the previous section, which by abuse
of notation we will continue to refer to as $\alpha$. 

Since $C$ is smooth, $K_{\infty}$ is isomorphic to $\F_{q}(\!(t)\!)$,
where $t$ is the choice of a uniformizer at $\infty$.

Next, we describe how to construct the vector bundle $E$ on $C$
that encodes the information of the lattice used in Browning and Vishe's
work. Fix $s$ and $\ell$, as well as $x^{(1)},\ldots,x^{(\ell)}\in P_{e-s,C}^{n+1}$
and $x^{(\ell+2)},\ldots,x^{(d-1)}\in P_{e,C}^{n+1}$, which $E$
will depend on.

Recall from Notation \ref{nota:psij} that $\Psi_{j}$ is a multilinear
form taking in as input a $d$-tuple of $(n+1)$-tuples of the form
$\left(x_{1}^{(i)},\ldots,x_{n+1}^{(i)}\right)\in P_{e,C}^{n+1}$.
In particular, if we fix all but the $(\ell+1)$st entry, we obtain
a linear form $\Psi_{j}^{\ell}$ that takes in as input a single $(n+1)$-tuple
in $P_{e,C}^{n+1}$ and outputs an element of $P_{(d-1)e,C}$. Pulling
back $P_{e,C}$ and $P_{(d-1)e,C}$ to $\Spec K_{\infty}$, we see
that $\Psi_{j}^{\ell}$ can be viewed as a linear form with coefficients
in $K_{\infty}$. Denote by $\gamma$ the $(n+1)\times(n+1)$ matrix
with entries in $K_{\infty}$ whose $j$th row corresponds to $\Psi_{j}^{\ell}$.

To construct the vector bundle on $C$, by the Beauville--Laszlo
theorem (Remark 2 of \cite{BL95}), it suffices to specify a vector
bundle on $C-\infty$, a vector bundle on $\widehat{\mathcal{O}_{C,\infty}}$,
and an isomorphism of these bundles on $\Spec K_{\infty}=\left(C-\infty\right)\times_{C}\Spec\widehat{\mathcal{O}_{C,\infty}}$.
Take $\mathcal{O}\left(D\right)|_{C-\infty}^{n+1}\oplus\mathcal{O}\left(K_{C}-D-\ell s\infty\right)|_{C-\infty}^{n+1}$
on $C-\infty$, $\widehat{\mathcal{O}_{C,\infty}}^{2n+2}$ on $\widehat{\mathcal{O}_{C,\infty}}$,
and the invertible transition matrix 

\[
\Lambda=\begin{bmatrix}\Id_{n+1} & 0\\
t^{-\ell s}\alpha\gamma & \Id_{n+1}
\end{bmatrix}\in K_{\infty}^{(2n+2)\times(2n+2)}.
\]

Let us call the corresponding vector bundle $E$. 
\begin{lem}
\label{lem:globalsectionswithKsl}Fix $x^{(1)},\ldots,x^{(\ell)}\in P_{e-s,C}^{n+1}$
and $x^{(\ell+2)},\ldots,x^{(d-1)}\in P_{e,C}^{n+1}$. Let $K_{s,\ell}(\alpha)$
be defined the same way as $N_{s,\ell}(\alpha)$, except with all
but the $(\ell+1)$st entry fixed, i.e. is the number of $x^{(\ell+1)}\in P_{e,C}^{n+1}$
such that $\alpha\left(\Psi_{j}\left(x^{(1)},\ldots,x^{(d-1)}\right)x\right)$
vanishes for all $x\in P_{e+\ell s,C}$ and all $j$. 

Then, $\#H^{0}\left(C,E\right)=\#K_{s,\ell}(\alpha).$ 
\begin{proof}
Observe that each global section of $E$ corresponds to pairs of sections
of $\mathcal{O}\left(D\right)|_{C-\infty}^{n+1}\oplus\mathcal{O}\left(K_{C}-D-\ell s\infty\right)|_{C-\infty}^{n+1}$
on $C-\infty$ and sections of $\widehat{\mathcal{O}_{C,\infty}}^{2n+2}$
on $\widehat{\mathcal{O}_{C,\infty}}$ such that after tensoring up
to $K_{\infty}$, they are equivalent under the transition matrix
$\Lambda$, i.e. $\begin{bmatrix}x\\
y
\end{bmatrix}\in\mathcal{O}\left(D\right)|_{C-\infty}^{n+1}\oplus\mathcal{O}\left(K_{C}-D-\ell s\infty\right)|_{C-\infty}^{n+1}$ such that $\begin{bmatrix}x\\
t^{-\ell s}\alpha\gamma x+t^{-\ell s}y
\end{bmatrix}\in\widehat{\mathcal{O}_{C,\infty}}^{2n+2}$, which is the same as saying $v_{\infty}(x)\ge0$ and $v_{\infty}\left(\alpha\gamma x+y\right)\ge\ell s$. 

It remains to verify that the number of such pairs $(x,y)$ is the
same as $\#K_{s,\ell}(\alpha)$. Consider the map $H^{0}\left(C,E\right)\to K_{s,\ell}(\alpha)$
sending pairs $(x,y)$ to $x$. 

Let's check this is well-defined. Given such a pair, it is clear that
$x\in P_{e,C}^{n+1}$. The existence of $y$ such that $v_{\infty}\left(\alpha\gamma x+y\right)\ge\ell s$
implies that $\alpha\gamma x$ can be written in the form $\beta'+\beta''$,
where $\beta'\in\mathcal{O}\left(K_{C}-D-\ell s\infty\right)|_{C-\infty}^{n+1}$
and $\beta''\in K_{\infty}\left(K_{C}-D-\ell s\infty\right)$ by taking
$\beta'=-y$ and $\beta''=\alpha\gamma x+y$. Consider the image of
$\alpha\gamma x\in K_{\infty}^{n+1}$ in $\left(K_{\infty}/K_{\infty}\left(K_{C}-D-\ell s\infty\right)+K(C)\left(K_{C}-D-\ell s\infty\right)|_{\ne\infty}\right)^{n+1}$.
By the same argument as the identification verifying (\ref{eq:lifttoKinfty})
(but with $D$ replaced with $D+\ell s\infty$), we have $K_{\infty}/K_{\infty}\left(K_{C}-D-\ell s\infty\right)+K(C)\left(K_{C}-D-\ell s\infty\right)|_{\ne\infty}\cong P_{e+\ell s,C}^{\vee}$,
so we may think of $\alpha\gamma x$ as an element of $\left(P_{e+\ell s,C}^{\vee}\right)^{n+1}$
(concretely, its action on an element of $P_{e+\ell s,C}^{n+1}$ is
given by the residue pairing---see Section 4 of \cite{tate68}).
Since $\alpha\gamma x$ can be expressed as $\beta'+\beta''$ above,
this implies that it is 0 when viewed as an element of $P_{e+\ell s,C}^{\vee}$. 

This also shows the map is surjective by reversing the steps: If $\alpha\gamma x$
is 0 as an element of $P_{e+\ell s,C}^{\vee}$, we can find $\beta'\in\mathcal{O}\left(K_{C}-D-\ell s\infty\right)|_{C-\infty}^{n+1}$
and $\beta''\in K_{\infty}\left(K_{C}-D-\ell s\infty\right)$ such
that $\alpha\gamma x=\beta'+\beta''$, so the pair $(x,-\beta')$
is in the preimage. 

Finally, we check that the map is injective, i.e. for each choice
of $x$, there is only one value of $y$ that satisfies $v_{\infty}\left(\alpha\gamma x+y\right)\ge\ell s$.
Indeed, if $y'$ also satisfies $v_{\infty}\left(\alpha\gamma x+y'\right)\ge\ell s$,
then $v_{\infty}\left(y-y'\right)\ge\ell s$ and $v_{p}\left(y-y'\right)\ge-v_{p}\left(K_{C}-D-\ell s\infty\right)$
for all $p\ne\infty$, so $y-y'\in H^{0}\left(C,\mathcal{O}\left(K_{C}-D\right)\right)=0$
because $\deg D=e>2g-2=\deg K_{C}$. 
\end{proof}
\end{lem}

An identical argument shows the following.
\begin{lem}
\label{lem:globalsectionswithKslII}Fix $x^{(1)},\ldots,x^{(\ell)}\in P_{e-s,C}^{n+1}$
and $x^{(\ell+2)},\ldots,x^{(d-1)}\in P_{e,C}^{n+1}$. Let $K'_{s,\ell}(\alpha)$
be defined the same way as $N_{s,\ell+1}(\alpha)$, except with all
but the $(\ell+1)$st entry fixed, i.e. is the number of $x^{(\ell+1)}\in P_{e-s,C}^{n+1}$
such that $\alpha\left(\Psi_{j}\left(x^{(1)},\ldots,x^{(d-1)}\right)x\right)$
vanishes for all $x\in P_{e+(\ell+1)s,C}$ and all $j$. 

Then, $\#H^{0}\left(C,E(-s\infty)\right)=\#K'_{s,\ell}(\alpha).$ 
\end{lem}

The following lemma is crucial for controlling the global sections
of $E$.
\begin{lem}
\label{lem:.relationshipwithdual}$E^{\vee}\otimes K_{C}\left(-\ell s\infty\right)\cong E$.
\begin{proof}
The vector bundle $E$ is constructed so that $E|_{C-\infty}=\mathcal{O}\left(D\right)|_{C-\infty}^{n+1}\oplus\mathcal{O}\left(K_{C}-D-\ell s\infty\right)|_{C-\infty}^{n+1}$
and $E|_{\widehat{\mathcal{O}_{C,\infty}}}=\widehat{\mathcal{O}_{C,\infty}}^{2n+2}$
with transition matrix 
\[
\Lambda=\begin{bmatrix}\Id_{n+1} & 0\\
t^{-\ell s}\alpha\gamma & \Id_{n+1}
\end{bmatrix}\in K_{\infty}^{(2n+2)\times(2n+2)}.
\]
Let us express this concisely as the data $\left(\mathcal{O}\left(D\right)|_{C-\infty}^{n+1}\oplus\mathcal{O}\left(K_{C}-D-\ell s\infty\right)|_{C-\infty}^{n+1},\widehat{\mathcal{O}_{C,\infty}}^{2n+2},\Lambda\right)$. 

By taking duals, this implies 
\begin{align*}
E^{\vee} & \cong\left(\mathcal{O}\left(-D\right)|_{C-\infty}^{n+1}\oplus\mathcal{O}\left(-K_{C}+D+\ell s\infty\right)|_{C-\infty}^{n+1},\widehat{\mathcal{O}_{C,\infty}}^{2n+2},\Lambda^{-T}\right)
\end{align*}
so that 
\begin{align*}
E^{\vee}\otimes K_{C}\left(-\ell s\infty\right) & \cong\left(\mathcal{O}\left(K_{C}-D-\ell s\infty\right)|_{C-\infty}^{n+1}\oplus\mathcal{O}\left(-D\right)|_{C-\infty}^{n+1},t^{\ell s\infty}\widehat{\mathcal{O}_{C,\infty}}^{2n+2},\Lambda^{-T}\right)\\
 & \cong\left(\mathcal{O}\left(K_{C}-D-\ell s\infty\right)|_{C-\infty}^{n+1}\oplus\mathcal{O}\left(-D\right)|_{C-\infty}^{n+1},\widehat{\mathcal{O}_{C,\infty}}^{2n+2},t^{-\ell s}\Lambda^{-T}\right),
\end{align*}
where the second isomorphism follows from the fact that $v_{\infty}\left(x\right)\ge\ell s$
is equivalent to $v_{\infty}\left(t^{-\ell s}x\right)\ge0.$

Observe that 
\[
t^{-\ell s}\Lambda^{-T}=\begin{bmatrix}t^{-\ell s}\Id_{n+1} & -t^{-\ell s}\alpha\gamma\\
0 & \Id_{n+1}
\end{bmatrix}=\begin{bmatrix}0 & I_{n+1}\\
-I_{n+1} & 0
\end{bmatrix}\Lambda\begin{bmatrix}0 & I_{n+1}\\
-I_{n+1} & 0
\end{bmatrix}^{-1},
\]
so we have 
\begin{align*}
 & E^{\vee}\otimes K_{C}\left(-\ell s\infty\right)\\
 & \cong\left(\mathcal{O}\left(K_{C}-D-\ell s\infty\right)|_{C-\infty}^{n+1}\oplus\mathcal{O}\left(-D\right)|_{C-\infty}^{n+1},\widehat{\mathcal{O}_{C,\infty}}^{2n+2},\begin{bmatrix}0 & I_{n+1}\\
-I_{n+1} & 0
\end{bmatrix}\Lambda\begin{bmatrix}0 & I_{n+1}\\
-I_{n+1} & 0
\end{bmatrix}^{-1}\right)\\
 & \cong\left(\mathcal{O}\left(-D\right)|_{C-\infty}^{n+1}\oplus\mathcal{O}\left(K_{C}-D-\ell s\infty\right)|_{C-\infty}^{n+1},\widehat{\mathcal{O}_{C,\infty}}^{2n+2},\Lambda\right)\\
 & \cong E,
\end{align*}
as desired.
\end{proof}
\end{lem}

Recall that if $E$ is a non-zero vector bundle, its \textit{slope}
is defined as 
\[
\mu\left(E\right)\coloneqq\frac{\deg\left(E\right)}{\rk\left(E\right)}.
\]
Moreover, $E$ is \textit{semi-stable} of slope $\lambda$ if $\mu\left(E\right)=\lambda$
and for every non-zero sub-bundle $E'\subset E$, we have $\mu\left(E'\right)\le\lambda$
(by convention, the zero vector bundle is semi-stable of any slope).
Equivalently, for any surjection of vector bundles $E\twoheadrightarrow E'$,
we have $\mu\left(E'\right)\ge\lambda$. 

It is a famous result (Proposition 1.3.9 of \cite{HNfiltration})
that every vector bundle $E$ has a unique \textit{Harder--Narasimhan
filtration}, i.e. a filtration by vector bundles $0=E_{0}\subset E_{1}\subset\cdots\subset E_{r}=E$
such that the following conditions are satisfied:
\begin{enumerate}
\item Each quotient $E_{i}/E_{i-1}$ is a semi-stable vector bundle of some
slope $\lambda_{i}$.
\item $\lambda_{1}>\cdots>\lambda_{r}$. 
\end{enumerate}
\begin{lem}
\label{lem:dualHN}Suppose $E$ is a vector bundle with Harder--Narasimhan
filtration 
\[
0=E_{0}\subset E_{1}\subset\cdots\subset E_{r}=E
\]
Then, the Harder--Narasimhan filtration of $E^{\vee}$ is given by
\[
0=\left(E_{r}/E_{r}\right)^{\vee}\subset\left(E_{r}/E_{r-1}\right)^{\vee}\subset\cdots\subset\left(E_{r}/E_{0}\right)^{\vee}=E^{\vee}.
\]
\end{lem}

\begin{proof}
Consider the short exact sequence 
\[
0\to E_{i+1}/E_{i}\to E_{r}/E_{i}\to E_{r}/E_{i+1}\to0.
\]
Dualizing this gives 
\[
0\to\left(E_{r}/E_{i+1}\right)^{\vee}\to\left(E_{r}/E_{i}\right)^{\vee}\to\left(E_{i+1}/E_{i}\right)^{\vee}\to0.
\]
Note that $\left(E_{i+1}/E_{i}\right)^{\vee}$ has the same rank as
$E_{i+1}/E_{i}$ and has degree negative that of $E_{i+1}/E_{i}$,
which means $\mu\left(\left(E_{i+1}/E_{i}\right)^{\vee}\right)=-\mu\left(E_{i+1}/E_{i}\right)$.
Moreover, it is also semi-stable since if $F\subset\left(E_{i}/E_{i+1}\right)^{\vee}$
is a sub-bundle, then $E_{i}/E_{i+1}\to F^{\vee}$ is a surjection
of vector bundles and semi-stability gives $-\mu(F)=\mu\left(F^{\vee}\right)\ge\mu\left(E_{i}/E_{i+1}\right),$
so $\mu\left(\left(E_{i}/E_{i+1}\right)^{\vee}\right)=-\mu\left(E_{i}/E_{i+1}\right)\ge\mu(F)$.
So by uniqueness of the Harder--Narasimhan filtration, this is indeed
the Harder--Narasimhan filtration of $E^{\vee}$.
\end{proof}
\begin{lem}
\label{lem:increaseslope}Suppose $E$ is a vector bundle on $C$
of slope $\lambda$. Then $E\left(r\infty\right)$ has slope $\lambda+r$.
Moreover, if $E$ is semi-stable, then so is $E\left(r\infty\right)$. 
\begin{proof}
The first part is clear, since rank doesn't change and the determinant
is given by $\left(\det E\right)^{\otimes1}\otimes\left(\det\mathcal{O}\left(r\infty\right)\right)^{\otimes\rk E}=\det E\otimes\mathcal{O}(r\rk E(\infty))$.
For the second, we're just tensoring by an invertible sheaf.
\end{proof}
\end{lem}

Let us use the notation of Lemma \ref{lem:dualHN} and return to the
problem. Since $E^{\vee}\otimes K_{C}\left(-\ell s\infty\right)\cong E$
by Lemma \ref{lem:.relationshipwithdual}, the corresponding quotients
of their Harder--Narasimhan filtrations must be isomorphic in a compatible
way, i.e. 
\begin{equation}
E_{i}\cong\left(E_{r}/E_{r-i}\right)^{\vee}\otimes K_{C}\left(-\ell s\infty\right)\label{eq:relationshippieceofdual}
\end{equation}
are isomorphic and the isomorphisms are compatible with the inclusions
between them. In particular, the slopes are also the same. Let $\lambda_{i}\coloneqq\mu\left(E_{i}/E_{i-1}\right)$. 

Then, the sequence of slopes $\lambda_{1}>\lambda_{2}>\cdots>\lambda_{r}$
and $-\lambda_{r}+2g-2-\ell s>-\lambda_{r-1}+2g-2-\ell s>\cdots>-\lambda_{1}+2g-2-\ell s$
are the same sequences, i.e. 
\[
\lambda_{i}=-\lambda_{r+1-i}+2g-2-\ell s.
\]
This means 
\[
\lambda_{\left\lceil r/2\right\rceil +1}<g-1-\frac{\ell s}{2}\text{ and }\lambda_{\left\lceil r/2\right\rceil }\ge g-1-\frac{\ell s}{2}.
\]

\begin{lem}
\label{lem:noglobal}Let $E$ be semi-stable and of slope $\lambda<0$.
Then $E$ has no global sections. 
\begin{proof}
Suppose $E$ has a global section, i.e. a (necessarily injective)
morphism $\mathcal{O}\to E$. By semi-stability of $E$, it follows
that $0=\mu\left(\mathcal{O}\right)\le\mu(E)$, which is a contradiction. 
\end{proof}
\end{lem}

Consider the short exact sequences 
\[
0\to E_{i}\to E_{i+1}\to E_{i+1}/E_{i}\to0.
\]
We have $h^{0}\left(C,E_{i+1}\right)\le h^{0}\left(C,E_{i}\right)+h^{0}\left(C,E_{i+1}/E_{i}\right)$
and $h^{0}\left(C,E_{i+1}\right)\ge h^{0}\left(C,E_{i}\right)$, so
\[
h^{0}\left(C,E\right)\le h^{0}\left(C,E_{r'}\right)+h^{0}\left(C,E_{r'+1}/E_{r'}\right)+\cdots+h^{0}\left(C,E_{r}/E_{r-1}\right)
\]
and 
\[
h^{0}\left(C,E\left(-s\infty\right)\right)\ge h^{0}\left(C,E_{r'}\left(-s\infty\right)\right)
\]
for any $r'$. 

Then, by inductively computing Euler characteristics and applying
Riemann--Roch for vector bundles on $C$, we have 

\begin{align*}
h^{0}\left(C,E\right) & \le h^{0}\left(C,E_{r'}\right)+h^{0}\left(C,E_{r'+1}/E_{r'}\right)+\cdots+h^{0}\left(C,E_{r}/E_{r-1}\right)\\
 & =h^{1}\left(C,E_{r'}\right)+\sum_{i=1}^{r'}\rk\left(E_{i}/E_{i-1}\right)\left(1-g+\lambda_{i}\right)+\sum_{i=r'+1}^{r}h^{0}\left(C,E_{i}/E_{i-1}\right).
\end{align*}

Similarly, we have 
\[
h^{0}\left(C,E\left(-s\infty\right)\right)\ge h^{1}\left(C,E_{r'}\left(-s\infty\right)\right)+\sum_{i=1}^{r'}\rk\left(E_{i}/E_{i-1}\right)\left(1-g+\lambda_{i}-s\right).
\]

\begin{lem}
Suppose $s\ge0$. Then $h^{1}\left(C,E_{r'}\right)\le h^{1}\left(C,E_{r'}\left(-s\infty\right)\right)$.
\label{lem:h1bound}
\end{lem}

\begin{proof}
To see this, note that we have the injection $E_{r'}\left(-s\infty\right)\hookrightarrow E_{r'}$,
and we want to show 
\[
H^{1}\left(C,E_{r'}\left(-s\infty\right)\right)\to H^{1}\left(C,E_{r'}\right)
\]
 is a surjection. Using (\ref{eq:relationshippieceofdual}), we can
write this map as 
\[
H^{1}\left(C,\left(E/E_{r-r'}\right)^{\vee}\otimes K_{C}\left(-(\ell+1)s\infty\right)\right)\to H^{1}\left(C,\left(E/E_{r-r'}\right)^{\vee}\otimes K_{C}\left(-\ell s\infty\right)\right),
\]
which by Serre duality we can write as 
\[
H^{0}\left(C,\left(E/E_{r-r'}\right)\left((\ell+1)s\infty\right)\right)^{\vee}\to H^{0}\left(C,\left(E/E_{r-r'}\right)\left(\ell s\infty\right)\right)^{\vee}.
\]
Taking duals, this is the same as the map 
\[
H^{0}\left(C,\left(E/E_{r-r'}\right)\left(\ell s\infty\right)\right)\to H^{0}\left(C,\left(E/E_{r-r'}\right)\left((\ell+1)s\infty\right)\right),
\]
which is evidently an injection, since $\left(E/E_{r-r'}\right)\left(\ell s\infty\right)\hookrightarrow\left(E/E_{r-r'}\right)\left((\ell+1)s\infty\right)$.
\end{proof}
Using the work above, we are finally able to compare $h^{0}\left(C,E\right)$
with $h^{0}\left(C,E(-s\infty)\right)$. See Lemma 6.4 of \cite{cohomology_circle}
for a similar statement when $g=0$. 
\begin{lem}
Suppose $s\ge\max\left(2g-1,2\right).$ Then,
\[
h^{0}\left(C,E\right)-h^{0}\left(C,E\left(-s\infty\right)\right)\le\begin{cases}
\left(n+1\right)s & \text{if }g=1\text{ or }\ell\ge1,\\
\left(n+1\right)\left(s+\frac{g+1}{2}\right) & \text{else}.
\end{cases}.
\]
\begin{proof}
Let us split into two cases, depending on the parity of $r$. 

Suppose $r$ is even, say $r=2u$, and using the notation right before
the lemma, say $r'=r/2=u$. Then, Lemma \ref{lem:h1bound} implies
\begin{align*}
 & h^{0}\left(C,E\right)-h^{0}\left(C,E\left(-s\infty\right)\right)\\
 & \le0+\sum_{i=1}^{r'}\rk\left(E_{i}/E_{i-1}\right)\left(\left(1-g+\lambda_{i}\right)-\left(1-g+\lambda_{i}-s\right)\right)+\sum_{i=r'+1}^{r}h^{0}\left(C,E_{i}/E_{i-1}\right)\\
 & =\sum_{i=1}^{r'}\rk\left(E_{i}/E_{i-1}\right)s+\sum_{i=r'+1}^{r}h^{0}\left(C,E_{i}/E_{i-1}\right)\\
 & \le\rk\left(E_{u}\right)s+\sum_{i=r'+1}^{r}h^{0}\left(C,E_{i}/E_{i-1}\right)\\
 & =(n+1)s+\sum_{i=r'+1}^{r}h^{0}\left(C,E_{i}/E_{i-1}\right).
\end{align*}
In the last step, note that $\rk E_{i}+\rk E_{r-i}=\rk E_{r}=2n+2$,
so $2\rk E_{u}=\rk E_{r}=2n+2$. 

We have $\lambda_{u+1}<g-1-\ell s/2,$ so for $g=1$ or $\ell\ge1$,
by Lemma \ref{lem:noglobal}, the sum $\sum_{i=r'+1}^{r}h^{0}\left(E_{i}/E_{i-1}\right)$
is simply 0. For $g\ge2$ and $\ell=0$, by Clifford's theorem for
vector bundles (Theorem 2.1 of \cite{brambila95}), we have 
\begin{align*}
\sum_{i=r'+1}^{r}h^{0}\left(E_{i}/E_{i-1}\right) & \le\sum_{i=r'+1}^{r}\rk\left(E_{i}/E_{i-1}\right)\left(1+\frac{\lambda_{i}}{2}\right)\\
 & <\rk\left(E/E_{u}\right)\left(1+\frac{g-1}{2}\right)\\
 & =\left(n+1\right)\left(\frac{g+1}{2}\right).
\end{align*}

Hence, for $r$ even, we have 
\[
h^{0}\left(C,E\right)-h^{0}\left(C,E\left(-s\infty\right)\right)\le\begin{cases}
\left(n+1\right)s & \text{if }g=1\text{ or }\ell\ge1,\\
\left(n+1\right)\left(s+\frac{g+1}{2}\right) & \text{else}.
\end{cases}.
\]

If $r$ is odd, say $r=2u+1$, then using the notation of the lemma,
say $r'=\left\lceil r/2\right\rceil =u+1$. 

For $\ell\ge1$, since $\lambda_{u+1}=g-1-\ell s/2<0$, the same argument
as above gives 
\begin{align*}
h^{0}\left(C,E\right)-h^{0}\left(C,E\left(-s\infty\right)\right) & \le\rk\left(E_{u}\right)s\le ns.
\end{align*}

For $\ell=0$, we have 
\begin{align*}
h^{0}\left(C,E\right) & \le h^{1}\left(C,E_{r'}\right)+\sum_{i=1}^{r'}\rk\left(E_{i}/E_{i-1}\right)\left(1-g+\lambda_{i}\right)+\sum_{i=r'+1}^{r}h^{0}\left(C,E_{i}/E_{i-1}\right)\\
 & \le h^{1}\left(C,E_{r'-1}\right)+h^{1}\left(C,E_{r'}/E_{r'-1}\right)+\sum_{i=1}^{r'}\rk\left(E_{i}/E_{i-1}\right)\left(1-g+\lambda_{i}\right)+\sum_{i=r'+1}^{r}h^{0}\left(C,E_{i}/E_{i-1}\right)
\end{align*}
and 
\[
h^{0}\left(C,E\left(-s\infty\right)\right)\ge h^{1}\left(C,E_{r'-1}\left(-s\infty\right)\right)+\sum_{i=1}^{r'-1}\rk\left(E_{i}/E_{i-1}\right)\left(1-g+\lambda_{i}-s\right).
\]
Using the lemma above, the fact that $\lambda_{u+1}=g-1$, and Serre
duality, we get 
\begin{align*}
h^{0}\left(C,E\right)-h^{0}\left(C,E\left(-s\infty\right)\right) & \le h^{1}\left(C,E_{u+1}/E_{u}\right)+\rk\left(E_{u}\right)s+\sum_{i=r'+1}^{r}h^{0}\left(C,E_{i}/E_{i-1}\right)\\
 & =h^{0}\left(C,\left(E_{u+1}/E_{u}\right)^{\vee}\otimes K_{C}\right)+\rk\left(E_{u}\right)s+\sum_{i=r'+1}^{r}h^{0}\left(C,E_{i}/E_{i-1}\right).
\end{align*}

Note that $\left(E_{u+1}/E_{u}\right)^{\vee}\otimes K_{C}$ has the
same rank as $E_{u+1}/E_{u}$, but has slope $-\lambda_{u+1}+2g-2=g-1$.
Since $E_{u+1}/E_{u}$ is semi-stable, $\left(E_{u+1}/E_{u}\right)^{\vee}\otimes K_{C}$
is semi-stable as well. 

Clifford's theorem for vector bundles tells us that 
\[
h^{0}\left(C,\left(E_{u+1}/E_{u}\right)^{\vee}\otimes K_{C}\right)\le\rk\left(E_{u+1}/E_{u}\right)+\frac{\rk\left(E_{u+1}/E_{u}\right)\left(g-1\right)}{2},
\]
so 
\begin{align*}
 & h^{0}\left(C,E\right)-h^{0}\left(C,E\left(-s\infty\right)\right)\\
 & \le\rk\left(E_{u+1}/E_{u}\right)+\frac{\rk\left(E_{u+1}/E_{u}\right)\left(g-1\right)}{2}+\rk\left(E_{u}\right)s+\sum_{i=r'+1}^{r}h^{0}\left(C,E_{i}/E_{i-1}\right)\\
 & =\rk\left(E_{u+1}/E_{u}\right)\left(\frac{g+1}{2}\right)+\rk\left(E_{u}\right)s+\sum_{i=u+2}^{r}h^{0}\left(C,E_{i}/E_{i-1}\right).
\end{align*}

We have $\lambda_{u+2}<g-1,$ so for $g=1$, by Lemma \ref{lem:noglobal},
the sum $\sum_{i=u+2}^{r}h^{0}\left(C,E_{i}/E_{i-1}\right)$ is simply
0. For $g\ge2$, by Clifford's theorem for vector bundles, we have 

\begin{align*}
\sum_{i=u+2}^{r}h^{0}\left(C,E_{i}/E_{i-1}\right) & =\sum_{i=u+2}^{r}\rk\left(E_{i}/E_{i-1}\right)\left(1+\frac{\lambda_{i}}{2}\right)<\rk\left(E/E_{u+1}\right)\left(1+\frac{g-1}{2}\right).
\end{align*}

If $u=0$, we have $E=E_{1}$ and $E_{0}=0$, so 
\begin{align*}
h^{0}\left(C,E\right)-h^{0}\left(C,E\left(-s\infty\right)\right) & \le\rk\left(E\right)\left(\frac{g+1}{2}\right)=\left(n+1\right)\left(g+1\right).
\end{align*}

If $u\ge1$, then $E_{u}\ne0$, so $\rk\left(E_{u}\right)\ge1$. Write
$a=\rk\left(E_{u}\right)\le n$. Then $\rk\left(E_{u+1}\right)=2n+2-a$,
so for $g=1$, we have 
\begin{align*}
h^{0}\left(C,E\right)-h^{0}\left(C,E\left(-s\infty\right)\right) & \le\left(2n+2-a-a\right)\left(\frac{g+1}{2}\right)+as\\
 & =2\left(n+1\right)+a\left(s-2\right)\\
 & \le ns+2
\end{align*}
(by assumption on $s$), and for $g\ge2$, we have 
\begin{align*}
h^{0}\left(C,E\right)-h^{0}\left(C,E\left(-s\infty\right)\right) & \le\left(2n+2-a-a\right)\left(\frac{g+1}{2}\right)+as+a\left(\frac{g+1}{2}\right)\\
 & =\left(n+1\right)\left(g+1\right)+a\left(s-\frac{g+1}{2}\right)\\
 & \le ns+\left(g+1\right)\left(\frac{n}{2}+1\right).
\end{align*}
The bounds for the even case are weaker for both $g=1$ and $g\ge2$,
and hence the result follows.
\end{proof}
\end{lem}

By chaining the lemma above, we arrive at a higher genus variant of
Davenport's shrinking lemma.
\begin{prop}
\label{prop:shrinking} Suppose $s\ge\max\left(2g-1,2\right)$. Then,
we have 
\[
\frac{N(\alpha)}{N_{s}\left(\alpha\right)}\le\begin{cases}
q^{\left(d-1\right)\left(n+1\right)s} & \text{if }g=1,\\
q^{\left(d-1\right)\left(n+1\right)s+\left(n+1\right)\left(g+1\right)/2} & \text{else}.
\end{cases}
\]
\begin{proof}
Recall that $E$ depends on $s$ and $\ell$, as well as $x^{(1)},\ldots,x^{(\ell)}\in P_{e-s,C}^{n+1}$
and $x^{(\ell+2)},\ldots,x^{(d-1)}\in P_{e,C}^{n+1}$. 

For $\ell=0$, let us first compare $N(\alpha)=N_{s,0}(\alpha)$ with
$N_{s,1}(\alpha)$. Recall that $N(\alpha)$ is the number of tuples
$\left(x^{(1)},\ldots,x^{(d-1)}\right)\in\left(P_{e,C}^{n+1}\right)^{d-1}$
such that $\alpha\left(\Psi_{j}\left(x^{(1)},\ldots,x^{(d-1)}\right)x\right)=0$
for $x\in P_{e,C}$ and all $j$. Similarly, $N_{s,1}(\alpha)$ is
the number of tuples $\left(x^{(1)},x^{(2)},\ldots,x^{(d-1)}\right)\in P_{e-s,C}\times\left(P_{e,C}^{n+1}\right)^{d-2}$
such that $\alpha\left(\Psi_{j}\left(x^{(1)},\ldots,x^{(d-1)}\right)x\right)=0$
for $x\in P_{e+s,C}$ and all $j$. Then, using Lemmas \ref{lem:globalsectionswithKsl}
and \ref{lem:globalsectionswithKslII}, as well as their notation,
observe that 
\[
\frac{N(\alpha)}{N_{s,1}(\alpha)}\le\frac{K_{s,0}(\alpha)}{K_{s,0}'(\alpha)}=\frac{\#H^{0}\left(C,E\right)}{\#H^{0}\left(C,E(-s\infty)\right)}=q^{h^{0}\left(C,E\right)-h^{0}\left(C,E(-s\infty)\right)}\le\begin{cases}
q^{(n+1)s} & \text{if }g=1,\\
q^{\left(n+1\right)\left(s+\frac{g+1}{2}\right)} & \text{else}.
\end{cases}
\]

Similarly, we have 
\[
\frac{N_{s,1}(\alpha)}{N_{s,2}(\alpha)},\ldots,\frac{N_{s,d-2}(\alpha)}{N_{s,d-1}(\alpha)=N_{s}(\alpha)}\le q^{(n+1)s}.
\]

Chaining these inequalities gives 

\[
\frac{N(\alpha)}{N_{s}\left(\alpha\right)}\le\begin{cases}
q^{\left(d-1\right)\left(n+1\right)s} & \text{if }g=1,\\
q^{\left(d-1\right)\left(n+1\right)s+\left(n+1\right)\left(g+1\right)/2} & \text{else}.
\end{cases}
\]
\end{proof}
\end{prop}

\section{Bounding $N_{s}\left(\alpha\right)$}

We continue using the same notation from the previous section. Next,
we'll choose a suitable $s$ so that $N_{s}(\alpha)$ is easier to
bound directly. In particular, we will be able to remove the dependence
on $\alpha$, so that we have 
\[
N_{s}(\alpha)=\#\left\{ \left(x^{(1)},\ldots,x^{(d-1)}\right)\in\left(P_{e-s,C}^{n+1}\right)^{d-1}:\Psi_{j}\left(x^{(1)},\ldots,x^{(d-1)}\right)=0\text{ for all }j\right\} .
\]
Then, a linear projection argument (exactly as in \cite{BrowningVisheRatCurves})
allows us to get a good bound on $N_{s}(\alpha)$. 

Depending on $s$, we can view each $\Psi_{j}$ as a function 
\[
\left(H^{0}\left(C,L\left(-s\infty\right)\right)^{n+1}\right)^{d-1}\to H^{0}\left(C,L\left(-s\infty\right)^{\otimes(d-1)}\right)
\]
and hence $\alpha\left(\Psi_{j}\left(x^{(1)},\cdots,x^{(d-1)}\right)-\right)$
as a function from $H^{0}\left(C,L\left(\left(d-1\right)s\infty\right)\right)$
to $\F_{q}$.

Suppose $\alpha$ factors through $Z$ minimally. 
\begin{lem}
\label{lem:getridofalpha}If 
\[
s>\max\left(\frac{\deg Z-e+2g-2}{d-1},e-\frac{\deg Z}{d-1}\right),
\]
then $\alpha\left(\Psi_{j}\left(x^{(1)},\cdots,x^{(d-1)}\right)-\right)=0$
implies $\Psi_{j}\left(x^{(1)},\cdots,x^{(d-1)}\right)=0$. 
\end{lem}

\begin{proof}
We get a long exact sequence \begin{center}
\begin{tikzcd}   H^{0}\left(C,L\left(-Z+(d-1)s\infty\right)\right) \rar & H^{0}\left(C,L\left((d-1)s\infty\right)\right) \rar              \ar[draw=none]{d}[name=X, anchor=center]{}     & H^{0}\left(Z,L|_{Z}\left((d-1)s\infty\right)\right) \ar[rounded corners,             to path={ -- ([xshift=2ex]\tikztostart.east)                       |- (X.center) \tikztonodes                       -| ([xshift=-2ex]\tikztotarget.west)                       -- (\tikztotarget)}]{dll}[at end]{} \\         H^{1}\left(C,L\left(-Z+(d-1)s\infty\right)\right) \rar & H^{1}\left(C,L\left((d-1)s\infty\right)\right) \rar & H^{1}\left(Z,L|_{Z}\left((d-1)s\infty\right)\right) \end{tikzcd} 
\end{center} Taking duals, we get \begin{center}
\begin{tikzcd}   H^{1}\left(Z,L|_{Z}\left((d-1)s\infty\right)\right)^\vee \rar & H^{1}\left(C,L\left((d-1)s\infty\right)\right)^\vee \rar              \ar[draw=none]{d}[name=X, anchor=center]{}     & H^{1}\left(C,L\left(-Z+(d-1)s\infty\right)\right)^\vee \ar[rounded corners,             to path={ -- ([xshift=2ex]\tikztostart.east)                       |- (X.center) \tikztonodes                       -| ([xshift=-2ex]\tikztotarget.west)                       -- (\tikztotarget)}]{dll}[at end]{} \\         H^{0}\left(Z,L|_{Z}\left((d-1)s\infty\right)\right)^\vee \rar & H^{0}\left(C,L\left((d-1)s\infty\right)\right)^\vee \rar & H^{0}\left(C,L\left(-Z+(d-1)s\infty\right)\right)^\vee \end{tikzcd} 
\end{center}

We have 
\[
s>\frac{\deg Z-e+2g-2}{d-1},
\]
i.e. $\deg Z<(d-1)s+e-2g+2$, which implies 
\[
H^{1}\left(C,L\left(-Z+(d-1)s\infty\right)\right)^{\vee}\cong H^{0}\left(C,L^{\vee}\left(Z-(d-1)s\infty+K_{C}\right)\right)
\]
(by Serre duality) is zero. Then, the above exact sequence beomes
the short exact sequence 
\[
0\to H^{0}\left(Z,L|_{Z}\left(\left(d-1\right)s\infty\right)\right)^{\vee}\to H^{0}\left(C,L\left(\left(d-1\right)s\infty\right)\right)^{\vee}\to H^{0}\left(C,L\left(-Z+\left(d-1\right)s\infty\right)\right)^{\vee}\to0.
\]

Note that $\alpha\left(\Psi_{j}\left(x^{(1)},\cdots,x^{(d-1)}\right)-\right)$
is an element of $H^{0}\left(C,L\left((d-1)s\infty\right)\right)^{\vee}$. 

Since $\alpha\sim Z$, we know $\alpha\left(\Psi_{j}\left(x^{(1)},\cdots,x^{(d-1)}\right)-\right)$
is an element of $H^{0}\left(Z,L|_{Z}\left((d-1)s\infty\right)\right)^{\vee}$
in the short exact sequence above.

If $\alpha\left(\Psi_{j}\left(x^{(1)},\cdots,x^{(d-1)}\right)-\right)$
(now viewed as an element of $H^{0}\left(Z,L|_{Z}\left((d-1)s\infty\right)\right)^{\vee}$)
is 0, then $\alpha$ factoring through $Z$ minimally means that it
doesn't vanish on any closed point of $Z$, which amounts to saying
that each $\Psi_{j}\left(x^{(1)}|_{Z},\cdots,x^{(d-1)}|_{Z}\right)x$
with $x\in H^{0}\left(C,L\left(\left(d-1\right)s\infty\right)\right)|_{Z}$
is 0. 

Since $H^{1}\left(C,L\left(-Z+(d-1)s\infty\right)\right)=0$, it follows
from the first long exact sequence above that $H^{0}\left(C,L\left(\left(d-1\right)s\infty\right)\right)\to H^{0}\left(Z,L|_{Z}\left(\left(d-1\right)s\infty\right)\right)$
is a surjection. In other words, we have $H^{0}\left(Z,L|_{Z}\left(\left(d-1\right)s\infty\right)\right)=H^{0}\left(C,L\left(\left(d-1\right)s\infty\right)\right)|_{Z}$.

Then, $\Psi_{j}\left(x^{(1)}|_{Z},\cdots,x^{(d-1)}|_{Z}\right)x$
vanishes on every $x\in H^{0}\left(Z,L|_{Z}\left(\left(d-1\right)s\infty\right)\right),$
which means that $\Psi_{j}\left(x^{(1)}|_{Z},\cdots,x^{(d-1)}|_{Z}\right)$
is zero. 

Finally, since 
\[
s>e-\frac{\deg Z}{d-1},
\]
i.e. $\deg Z>(d-1)(e-s)$, we have $H^{0}\left(C,L\left(-s\infty\right)^{\otimes(d-1)}\left(-Z\right)\right)=0$.
Using the exact sequence 
\[
0\to H^{0}\left(C,L\left(-s\infty\right)^{\otimes(d-1)}\left(-Z\right)\right)\to H^{0}\left(C,L\left(-s\infty\right)^{\otimes(d-1)}\right)\to H^{0}\left(Z,L|_{Z}\left(-s\infty\right)^{\otimes(d-1)}\right),
\]
it follows that 
\[
H^{0}\left(C,L\left(-s\infty\right)^{\otimes(d-1)}\right)\to H^{0}\left(Z,L|_{Z}\left(-s\infty\right)^{\otimes(d-1)}\right)
\]
is injective. Under this map $\Psi_{j}\left(x^{(1)},\cdots,x^{(d-1)}\right)$
maps to $\Psi_{j}\left(x^{(1)}|_{Z},\cdots,x^{(d-1)}|_{Z}\right),$
which is zero, so we conclude that $\Psi_{j}\left(x^{(1)},\cdots,x^{(d-1)}\right)$
is zero as well.
\end{proof}
So we take 
\begin{equation}
s=\max\left(\left\lfloor \frac{\deg Z-e+2g-2}{d-1}\right\rfloor ,\left\lfloor e-\frac{\deg Z}{d-1}\right\rfloor ,2g-2,1\right)+1.\label{eq:choiceofs}
\end{equation}
Note that the first term dominates the second iff $\deg Z>de/2-g+1$.
The terms $2g-2$ and 1 come from Proposition \ref{prop:shrinking}. 

For this choice of $s$, we then have 
\[
N_{s}\left(\alpha\right)=\#\left\{ \left(x^{(1)},\ldots,x^{(d-1)}\right)\in\left(P_{e-s,C}^{n+1}\right)^{d-1}:\Psi_{j}\left(x^{(1)},\ldots,x^{(d-1)}\right)=0\text{ for all }j\right\} .
\]

By a standard argument (see the two paragraphs preceding Lemma 5.3
in \cite{BrowningVisheRatCurves}), the system of equations $\left\{ \Psi_{j}=0\right\} _{j}$
defines an affine variety $V$ in $\A_{K(C)}^{\left(d-1\right)\left(n+1\right)}$
of dimension at most $(d-2)(n+1)$. 

Indeed, by definition, the intersection of $V$ with the diagonal
$\Delta\cong\A_{K(C)}^{n+1}$ in $\A_{K(C)}^{\left(d-1\right)\left(n+1\right)}$
is contained in the singular locus of the smooth hypersurface cut
out by $f.$ So 
\[
0\ge\dim V+\dim\Delta-(d-1)(n+1)=\dim V-(d-2)(n+1).
\]
Lemma 2.8 from \cite{browningvishecubic15} in the higher genus setting
is the following (see also Lemma 29 in \cite{terminal}). 
\begin{prop}
\label{prop:nsbound}For the above choice of $s$,
\[
N_{s}\left(\alpha\right)=O_{d,n}\left(q^{(d-2)(n+1)h^{0}(C,L(-s\infty))}\right).
\]
\end{prop}

\section{Putting everything together}

Recall that to prove Theorem \ref{thm:alternative}, it suffices to
check the sum over minor arcs goes to 0, namely
\[
q^{-(n+1)(e+1-g)}\sum_{\alpha:\deg(\alpha)>e-2g+1}S(\alpha)\to0.
\]

By combining Propositions \ref{prop:weyldifference}, \ref{prop:shrinking},
and \ref{prop:nsbound}, it follows that 
\[
S(\alpha)\le\left(\#P_{e,C}^{n+1}\right)^{1-\frac{d-1}{2^{d-1}}}O_{d,n}\left(q^{\frac{\left(n+1\right)\left((d-1)s+f(g)\right)+h^{0}(C,L(-s\infty))(d-2)(n+1)}{2^{d-1}}}\right)
\]
for the specific choice of $s$ used in the previous section (namely
from (\ref{eq:choiceofs})), and 
\[
f(g)=\begin{cases}
0 & \text{if }g=1,\\
\frac{g+1}{2} & \text{else}.
\end{cases}
\]
Then, 
\begin{align*}
 & \lim_{q\to\infty}q^{-(n+1)(e+1-g)}S(\alpha)\\
 & =\lim_{q\to\infty}q^{-(n+1)(e+1-g)}\sum_{\deg(\alpha)>e-2g+1}\left(\#P_{e,C}^{n+1}\right)^{1-\frac{d-1}{2^{d-1}}}O_{d,n}\left(q^{\frac{\left(n+1\right)\left((d-1)s+f(g)\right)+h^{0}(C,L(-s\infty))(d-2)(n+1)}{2^{d-1}}}\right)\\
 & =\lim_{q\to\infty}\sum_{\deg(\alpha)>e-2g+1}O_{d,n}\left(q^{(n+1)(e+1-g)\left(-\frac{d-1}{2^{d-1}}\right)+\frac{\left(n+1\right)\left((d-1)s+f(g)\right)+h^{0}(C,L(-s\infty))(d-2)(n+1)}{2^{d-1}}}\right).
\end{align*}
Note that the number of $\alpha$ such that $\deg(\alpha)=i$ is equal
to some function of $i$ times $q^{i}$ by Theorem \ref{thm:LW},
and, evidently, $q^{\deg Z}$ is an upper bound on the number of $\alpha$
that factor through a specific $Z$. So, it suffices to show that
the following quantity is $<0$: 

\begin{equation}
2\deg Z+\left(-\frac{d-1}{2^{d-1}}\right)(e+1-g)(n+1)+\frac{\left(n+1\right)\left((d-1)s+f(g)\right)+h^{0}(C,L(-s\infty))(d-2)(n+1)}{2^{d-1}}.\label{eq:boundthis}
\end{equation}

We'll continue our analysis depending on each of the possible tuples
of conditions on $g,d,e,$ and $n$ as in Theorem \ref{thm:main}.
Recall that 
\[
s=\max\left(\left\lfloor \frac{\deg Z-e+2g-2}{d-1}\right\rfloor ,\left\lfloor e-\frac{\deg Z}{d-1}\right\rfloor ,2g-2,1\right)+1.
\]
and that the first term dominates the first when $\deg Z>de/2-g+1$
and the second dominates the first when $\deg Z\le de/2-g+1$. 

\subsection{$d\ge4,g=1,e\ge2^{d}\left(d-1\right)^{2}+1,$ and $n\ge2^{d}\left(d-1\right)+1$}

For $e-2g+2\le\deg Z\le\frac{de}{2}-g+1$, note that $\left\lfloor e-\frac{\deg Z}{d-1}\right\rfloor \ge\left\lfloor e-\frac{\frac{de}{2}-g+1}{d-1}\right\rfloor \ge\left\lfloor \frac{\left(\frac{d}{2}-1\right)e}{d-1}\right\rfloor \ge\max\left(2g-2,1\right),$
so we have $s=\left\lfloor e-\frac{\deg Z}{d-1}\right\rfloor +1$
and hence 
\[
e-s=\left\lceil \frac{\deg Z}{d-1}\right\rceil -1.
\]
We have $e-s=\left\lceil \frac{\deg Z}{d-1}\right\rceil -1\ge\left\lceil \frac{e-2g+2}{d-1}\right\rceil -1,$
which, by assumption on $e$, is larger than $2g-2$. 

For $\frac{de}{2}-g+1<\deg Z\le\frac{de}{2}+1$, note that $\left\lfloor \frac{\deg Z-e+2g-2}{d-1}\right\rfloor \ge\left\lfloor \frac{\frac{de}{2}-g+1-e+2g-2}{d-1}\right\rfloor \ge\left\lfloor \frac{\left(\frac{d}{2}-1\right)e}{d-1}\right\rfloor \ge2g-2$,
we have $s=\left\lfloor \frac{\deg Z-e+2g-2}{d-1}\right\rfloor +1$
and hence 
\[
e-s=\left\lceil \frac{de-\deg Z-2g+2}{d-1}\right\rceil -1.
\]
We have $e-s=\left\lceil \frac{de-\deg Z-2g+2}{d-1}\right\rceil -1\ge\left\lceil \frac{\frac{de}{2}-1-2g+2}{d-1}\right\rceil -1,$
which, by assumption on $e$, is larger than $2g-2$. 

In this case, we have $f(g)=0$, and hence $\eqref{eq:boundthis}$
can be re-written to say that we need $n$ to be least 
\begin{equation}
\frac{2^{d}\deg Z}{(d-1)(e-g+1)-(d-1)s-(d-2)h^{0}\left(C,L(-s\infty)\right)}.\label{eq:fraction-1}
\end{equation}

Since $e-s>2g-2$, it follows that $h^{0}\left(C,L(-s\infty)\right)=e-s-g+1$.
The denominator of the fraction in (\ref{eq:fraction-1}) is then 

\[
e-s-g+1=\min\left(\left\lceil \frac{\deg Z}{d-1}\right\rceil ,\left\lceil \frac{de-\deg Z-2g+2}{d-1}\right\rceil \right)-g.
\]

For $e-2g+2\le\deg Z\le\frac{de}{2}-g+1$, we require 
\[
n\ge\frac{2^{d}\deg Z}{\left\lceil \frac{\deg Z}{d-1}\right\rceil -g}.
\]
If we write $\deg Z=a(d-1)-r$ with $0\le r\le d-2$, then the right-hand
side is largest when $r=0$. Moreover, the function $\frac{a(d-1)}{a-g}$
is decreasing in $a$, so we must have 
\[
n\ge\frac{2^{d}(d-1)\left\lceil \frac{e-2g+2}{d-1}\right\rceil }{\left\lceil \frac{e-2g+2}{d-1}\right\rceil -g}.
\]

For $\frac{de}{2}-g+1<\deg Z\le\frac{de}{2}+1,$ we have 

\[
n\ge\frac{2^{d}\deg Z}{\left\lceil \frac{de-\deg Z-2g+2}{d-1}\right\rceil -g}
\]
and $\frac{2^{d}\deg Z}{\frac{de-\deg Z-2g+2}{d-1}-g}\ge\frac{2^{d}\deg Z}{\left\lceil \frac{de-\deg Z-2g+2}{d-1}\right\rceil -g}.$
To see that the $e-2g+2\le\deg Z\le\frac{de}{2}-g+1$ range dominates,
we claim 
\[
\frac{2^{d}\left(\frac{de}{2}+1\right)}{\frac{\frac{de}{2}-2g+1}{d-1}-g}\le\frac{2^{d}(d-1)A}{A-g},
\]
where $A=\left\lceil \frac{e-2g+2}{d-1}\right\rceil $. Rewriting
this, we have 
\begin{equation}
A\le\frac{\frac{de}{2}+1}{d+1}.\label{eq:asdf}
\end{equation}
To check this, it suffices to show 
\[
e\le\frac{(d-1)d\left(\frac{e}{2}-1\right)}{d+1}=\frac{d(d-1)}{2(d+1)}e-\frac{d(d-1)}{d+1}.
\]
The right-hand side is linear in $e$ with larger slope, and so it
suffices to check it at the smallest value of $e$. For the chosen
range of $e$, this is evidently true. 

So it remains to show 
\[
2^{d}\left(d-1\right)+1\ge\frac{2^{d}(d-1)\left\lceil \frac{e}{d-1}\right\rceil }{\left\lceil \frac{e}{d-1}\right\rceil -1}
\]
for $e\ge2^{d}\left(d-1\right)^{2}+1$. But this is equivalent to
showing 
\[
A\ge2^{d}\left(d-1\right)+1,
\]
i.e. 
\[
e\ge2^{d}\left(d-1\right)^{2}+1.
\]

\subsection{$d\ge4,g\ge2,e\ge2^{d-1}\left(d-1\right)^{2}\left(3g+1\right)+d\left\lceil \frac{3g-1}{2}\right\rceil +\left\lfloor \frac{g-1}{2}\right\rfloor ,$
and $n\ge2^{d}\left(d-1\right)+1$}

The analysis is identical to the previous case except that we now
have $f(g)=\left(g+1\right)/2$ instead of 0. $\eqref{eq:boundthis}$
can be re-written to say that we need $n$ to be least 
\begin{equation}
\frac{2^{d}\deg Z}{(d-1)(e-g+1)-(d-1)s-(d-2)h^{0}\left(C,L(-s\infty)\right)-\frac{g+1}{2}}.\label{eq:fraction-1-2}
\end{equation}
Since $e-s>2g-2$, it follows that $h^{0}\left(C,L(-s\infty)\right)=e-s-g+1$.
The denominator of the fraction in (\ref{eq:fraction-1-2}) is then 

\[
e-s-g+1-\frac{g+1}{2}=\min\left(\left\lceil \frac{\deg Z}{d-1}\right\rceil ,\left\lceil \frac{de-\deg Z-2g+2}{d-1}\right\rceil \right)-\frac{3g+1}{2}.
\]
For $e-2g+2\le\deg Z\le\frac{de}{2}-g+1$, we require 
\[
n\ge\frac{2^{d}\deg Z}{\left\lceil \frac{\deg Z}{d-1}\right\rceil -\frac{3g+1}{2}}.
\]
If we write $\deg Z=a(d-1)-r$ with $0\le r\le d-2$, then the right-hand
side is largest when $r=0$. Moreover, the function $\frac{a(d-1)}{a-\frac{3g+1}{2}}$
is decreasing in $a$, so we must have 
\[
n\ge\frac{2^{d}(d-1)\left\lceil \frac{e-2g+2}{d-1}\right\rceil }{\left\lceil \frac{e-2g+2}{d-1}\right\rceil -\frac{3g+1}{2}}.
\]

For $\frac{de}{2}-g+1<\deg Z\le\frac{de}{2}+1,$ we have 

\[
n\ge\frac{2^{d}\deg Z}{\left\lceil \frac{de-\deg Z-2g+2}{d-1}\right\rceil -\frac{3g+1}{2}}
\]
and $\frac{2^{d}\deg Z}{\frac{de-\deg Z-2g+2}{d-1}-\frac{3g+1}{2}}\ge\frac{2^{d}\deg Z}{\left\lceil \frac{de-\deg Z-2g+2}{d-1}\right\rceil -\frac{3g+1}{2}}.$
To see that the $e-2g+2\le\deg Z\le\frac{de}{2}-g+1$ range dominates,
we claim 
\[
\frac{2^{d}\left(\frac{de}{2}+1\right)}{\frac{\frac{de}{2}-2g+1}{d-1}-\frac{3g+1}{2}}\le\frac{2^{d}(d-1)A}{A-\frac{3g+1}{2}},
\]
where $A=\left\lceil \frac{e-2g+2}{d-1}\right\rceil $. Rewriting
this, we have 
\[
A\le\frac{\frac{de}{2}+1}{d+\frac{g-1}{3g+1}}.
\]
But 
\[
\frac{\frac{de}{2}+1}{d+1}\le\frac{\frac{de}{2}+1}{d+\frac{g-1}{3g+1}},
\]
so the same argument (namely (\ref{eq:asdf})) in the previous case
shows that this is true for the chosen range of $e$.

So it remains to show 
\[
2^{d}\left(d-1\right)+1\ge\frac{2^{d}(d-1)\left\lceil \frac{e-2g+2}{d-1}\right\rceil }{\left\lceil \frac{e-2g+2}{d-1}\right\rceil -\frac{3g+1}{2}}
\]
for $e\ge2^{d-1}\left(d-1\right)^{2}\left(3g+1\right)+d\left\lceil \frac{3g-1}{2}\right\rceil +\left\lfloor \frac{g-1}{2}\right\rfloor $.
But this is equivalent to showing 
\[
A\ge\left(\frac{3g+1}{2}\right)\left(2^{d}\left(d-1\right)+1\right),
\]
i.e. 
\begin{align*}
e & \ge\left(d-1\right)\left(\left(3g+1\right)2^{d-1}\left(d-1\right)+\left\lceil \frac{3g+1}{2}\right\rceil \right)-\left(d-2\right)+\left(2g-2\right)\\
 & =2^{d-1}\left(d-1\right)^{2}\left(3g+1\right)+d\left\lceil \frac{3g-1}{2}\right\rceil +\left\lfloor \frac{g-1}{2}\right\rfloor .
\end{align*}

\subsection{$d=3,g=1,e\ge45,$ and $n\ge17$}

Similar to the previous case, since $\left\lfloor \frac{\left(\frac{d}{2}-1\right)e}{d-1}\right\rfloor \ge\max\left(2g-2,1\right)$,
the assumptions imply that 
\[
n\ge\frac{8\cdot2\cdot\left\lceil \frac{e-2g+2}{2}\right\rceil }{\left\lceil \frac{e-2g+2}{2}\right\rceil -g}
\]
for $e-2g+2\le\deg Z\le\frac{3e}{2}-g+1$ and 
\[
n\ge\frac{8\deg Z}{\left\lceil \frac{3e-\deg Z-2g+2}{2}\right\rceil -g}
\]
for $\frac{3e}{2}-g+1<\deg Z\le\frac{3e}{2}+1.$ In particular, we
have 
\[
n\ge8\max\left(\frac{2\left\lceil \frac{e}{2}\right\rceil -2g+2}{\left\lceil \frac{e}{2}\right\rceil -2g+1},\frac{\deg Z}{\left\lceil \frac{3e-\deg Z}{2}\right\rceil -2g+1}\right).
\]
Note that 
\[
\frac{\deg Z}{\left\lceil \frac{3e-\deg Z}{2}\right\rceil -2g+1}\le\frac{\deg Z}{\frac{3e-\deg Z}{2}-2g+1}\le\frac{\frac{3e}{2}+1}{\frac{3e/2-1}{2}-2g+1}=\frac{6e+4}{3e-2-8g+4},
\]
and we claim 
\[
\frac{2\left\lceil \frac{e}{2}\right\rceil -2g+2}{\left\lceil \frac{e}{2}\right\rceil -2g+1}\le\frac{6e+4}{3e-2-8g+4}.
\]
This is equivalent to 
\[
\left(\left\lceil \frac{e}{2}\right\rceil -g+1\right)\left(3e+2-8g\right)\le\left(3e+2\right)\left(\left\lceil \frac{e}{2}\right\rceil -2g+1\right),
\]
which we can rewrite as 
\[
3e+2\le8\left(\left\lceil \frac{e}{2}\right\rceil -g+1\right).
\]
But the right-hand side is at least $4e-8g+8$, which is at least
$3e+2$ by our assumption on $e$. So it remains to check that 

\[
17\ge16\cdot\frac{3e+2}{3e-2-8g+4}.
\]
for $e\ge45$, which is easy to verify. 

\subsection{$d=3,g\ge2,e\ge170g/3+32/3,$ and $n\ge17$}

Similar to the previous case, the assumptions imply that 
\[
n\ge\frac{8\cdot2\cdot\left\lceil \frac{e-2g+2}{2}\right\rceil }{\left\lceil \frac{e-2g+2}{2}\right\rceil -\frac{3g+1}{2}}
\]
for $e-2g+2\le\deg Z\le\frac{3e}{2}-g+1$ and 
\[
n\ge\frac{8\deg Z}{\left\lceil \frac{3e-\deg Z-2g+2}{2}\right\rceil -\frac{3g+1}{2}}
\]
for $\frac{3e}{2}-g+1<\deg Z\le\frac{3e}{2}+1.$ In particular, we
have 
\[
n\ge16\max\left(\frac{2\left\lceil \frac{e}{2}\right\rceil -2g+2}{2\left\lceil \frac{e}{2}\right\rceil -5g+1},\frac{\deg Z}{2\left\lceil \frac{3e-\deg Z}{2}\right\rceil -5g+1}\right).
\]
Note that 
\[
\frac{\deg Z}{2\left\lceil \frac{3e-\deg Z}{2}\right\rceil -5g+1}\le\frac{\deg Z}{3e-\deg Z-5g+1}\le\frac{\frac{3e}{2}+1}{\frac{3e}{2}-1-5g+1}=\frac{3e+2}{3e-10g},
\]
and we claim 
\[
\frac{2\left\lceil \frac{e}{2}\right\rceil -2g+2}{2\left\lceil \frac{e}{2}\right\rceil -5g+1}\le\frac{3e+2}{3e-10g}.
\]
This is equivalent to 
\[
\left(2\left\lceil \frac{e}{2}\right\rceil -2g+2\right)\left(3e-10g\right)\le\left(3e+2\right)\left(2\left\lceil \frac{e}{2}\right\rceil -5g+1\right),
\]
which we can rewrite as 
\[
9ge+20g^{2}-10g+3e-2\le\left(20g+4\right)\left\lceil \frac{e}{2}\right\rceil .
\]
But the right-hand side is at least $10ge+2e$, which is evidentally
larger than the left-hand side for the chosen ranges of $e$ and $g$.
So it remains to check that 

\[
17\ge16\cdot\frac{3e+2}{3e-10g}.
\]
for $e\ge170g/3+32/3$, which is easy to verify. 

\subsection{$d=2,g=1,e\ge14,$ and $n\ge5$}

For $e-2g+2\le\deg Z\le\frac{de}{2}-g+1$, note that only $\deg Z=e$
is the only possibility. Hence, in this range, $s=1+1=2.$ 

For $\frac{de}{2}-g+1<\deg Z\le\frac{de}{2}+1$, note that $\deg Z=e+1$
is the only possibility. Hence, similarly, we have $s=2$ in this
range. 

$\eqref{eq:boundthis}$ can be re-written to say that we need $n$
to be least 

\[
\frac{4\deg Z}{e-2}.
\]
Since $\deg Z\le e+1$, it remains to verify that 
\[
5\ge\frac{4e+4}{e-2},
\]
which is true for $e\ge14$.

\subsection{$d=2,g\ge2,e\ge35g/2-15/2,$ and $n\ge5$}

For $e-2g+2\le\deg Z\le\frac{de}{2}-g+1=e-g+1$, note that $e-\deg Z\le2g-2.$
Hence, in this range, $s=2g-1,$ so $e-s=e-2g+1.$ 

For $e-g+1=\frac{de}{2}-g+1<\deg Z\le\frac{de}{2}=e$, note that $\deg Z-e+2g-2\le2g-2$.
Hence, in this range, $s=2g-1$. 

For $\deg Z=\frac{de}{2}+1=e+1$, we have $\deg Z-e+2g-2=2g-1$, so
we take $s=2g$. 

$\eqref{eq:boundthis}$ can be re-written to say that we need $n$
to be least 

\[
\frac{4\deg Z}{e-g+1-s-\frac{g+1}{2}}.
\]
For the range $e-2g+2\le\deg Z\le e$, we have $s=2g-1$, so 
\[
\frac{4\deg Z}{e-g+1-s-\frac{g+1}{2}}\le\frac{8e}{2e-7g+3}.
\]
For $\deg Z=e+1$, we have $s=2g-2$, so 
\[
\frac{4\deg Z}{e-g+1-s-\frac{g+1}{2}}\le\frac{8e+8}{2e-7g+5}.
\]
For the chosen range of $e$ and $g$, it is evident that 
\[
5\ge\frac{8e}{2e-7g+3}\ge\frac{8e+8}{2e-7g+5},
\]
as desired.

This completes the proof of Theorem \ref{thm:main}.

\section{Application to the Fujita invariant}

In this section, we prove Corollary \ref{cor:anumber}.

Let $X\subset\P_{\C}^{n}$ be a smooth hypersurface satisfying the
conditions of the corollary. Our goal is to show that no proper subvariety
$V\subset X$ satisfies $a\left(V,-K_{X}|_{V}\right)\ge a\left(X,-K_{X}\right)=1$
(recall Definition \ref{def:fujita}). 

So, for the sake of contradiction, suppose $a\left(V,-K_{X}|_{V}\right)\ge1$
for some $V$. 

First, for any open subset $U\subset Y$ a variety and $C$ a smooth
projective curve, define $\Mor_{U}\left(C,Y\right)$ to be the space
of maps $C\to Y$ such that the image of $C$ intersects $U$. 
\begin{lem}
\label{lem:lowerboundim}Let $Y$ be a smooth variety of dimension
$n'$ and $f\colon C\to Y$ be a morphism. Suppose $-f^{*}K_{Y}\cdot C+\left(1-g(C)\right)n'-n'\ge0.$
Then, there is a non-empty open subset $U\subset Y$ with the image
of $f$ intersecting $U$, such that $\dim_{[f]}\Mor_{U}\left(C,Y\right)\ge-f^{*}K_{Y}\cdot C+\left(1-g(C)\right)n'$. 
\begin{proof}
For a fixed closed point $x\in C$, consider the morphism $\Mor\left(C,Y\right)\to Y$
that sends a map to the image of $x$ in $Y$. We can similarly define
$\Mor\left(C,Y,x\mapsto u\right)$ and $\Mor\left(C,Y,x\mapsto U\right)$
to denote maps that send $x$ to a fixed point $u\in U$ and those
that send $x$ to some point in $U$, respectively. We then have the
following two Cartesian diagrams:

\[\begin{tikzcd}          {\Mor(C,Y,x\mapsto u)} \ar[r] \ar[d] \ar[dr, phantom, "\square"] &  {\Mor(C,Y,x\mapsto U)} \ar[d] \ar[dr, phantom, "\square"] \ar[r] & {\Mor(C,Y)}\ar[d]  \\         u \ar[r]& U \ar[r] & Y     \end{tikzcd}\]Since
$-f^{*}K_{Y}\cdot C+\left(1-g(C)\right)n'-n'\ge0$, general lower
bounds on dimension for these moduli spaces (see Theorem 1.7 of \cite{Kollar_Rational_Curves})
implies for any $u\in Y$, every irreducible component of $\Mor\left(C,Y,x\mapsto u\right)$
at $[f]$ has dimension at least $-f^{*}K_{Y}\cdot C+\left(1-g(C)\right)n'-n'\ge0$,
i.e. is non-empty. This implies that $\Mor\left(C,Y\right)\to Y$
is dominant. Then, a standard fiber dimension theorem implies that
there is some dense open subset $U\subset Y$ so that 
\[
\dim_{[f]}\Mor\left(C,Y,x\mapsto U\right)=\dim U+\dim_{[f]}\Mor\left(C,Y,x\mapsto u\right).
\]
So 
\begin{align*}
\dim_{[f]}\Mor_{U}\left(C,Y\right) & \ge\dim_{[f]}\Mor\left(C,Y,x\mapsto U\right)\\
 & \ge n'-f^{*}K_{Y}\cdot C+\left(1-g\left(C\right)\right)n'-n'\\
 & =-f^{*}K_{Y}\cdot C+\left(1-g\left(C\right)\right)n',
\end{align*}
where in the second line we again use the lower bound on dimension
for these mapping spaces.
\end{proof}
\end{lem}

Now, suppose $Y\to V$ is a resolution of singularities. Then, $a\left(Y,-K_{X}|_{Y}\right)=a\left(V,-K_{X}|_{V}\right)\ge1$.
This implies $K_{Y}-K_{X}|_{Y}+\frac{1}{m}K_{X}|_{Y}$ is not pseudo-effective
for any integer $m>0$. By Theorem 0.2 of \cite{BDPP}, for a general
point $x\in Y$, we have a smooth projective curve $C_{m}$ passing
through $x$ such that 
\[
\left(-K_{Y}+K_{X}|_{Y}-\frac{1}{m}K_{X}|_{Y}\right)\cdot C_{m}>0,
\]
i.e. 
\begin{align}
\left(-K_{Y}+K_{X}|_{Y}\right)\cdot C_{m} & >\frac{1}{m}K_{X}|_{Y}\cdot C_{m}=\frac{1}{m}e_{m}(d-n-1)\label{eq:boundfromfujitaassump}
\end{align}
for some $e_{m}>0$ depending on $m$. 

Since we don't know anything about $e_{m}$, the next trick is to
use a combination of Frobenius pullbacks (which increase degree but
not genus) and Artin--Schreier covers (which increase both genus
and degree) to obtain good control over the degree. 

Let us fix a few seemingly arbitrary quantities. 

Choose $m=2\cdot4^{d}\left(n+1-d\right)+1$. 

Now, consider some dominant map of smooth projective curves $C\to C_{m}$
of any degree and $C$ of genus $g(C)$ larger than $2\cdot4^{d}\left(n+1-d\right)$,
such that the composition $C\to C_{m}$ with $C_{m}\to Y$ has degree
at least 2. Let this composition be called $h$, which is a morphism
from $C$ to $Y$. 

By abuse of notation, suppose the degree of $h$ is $e_{m}$. By the
same spreading out argument as in Section \ref{sec:Spreading-out-and},
we can assume we are working over characteristic $p$ with $p$ sufficiently
large. Specifically, choose $p$ large enough so that 
\[
2\cdot4^{d}\left(p-1\right)\left(n+1-d\right)+2\cdot4^{d}\left(n+1-d\right)+1<pg\left(C\right)
\]
(here, we are using the fact that $g\left(C\right)>2\cdot4^{d}\left(n+1-d\right)$). 

Consider $C'\to C$ an Artin--Schreier cover given by $y^{p}-y=u$
with $u\in K(C)$ chosen so that there is exactly one degree one point
$x\in C$ for which $v_{x}\left(u\right)=-m_{x}$ with $p\nmid m_{x}$
and $m_{x}$ a positive integer, and $v_{x'}(u)\ge0$ for all other
closed points $x'$ by Riemann--Roch. Then, by Proposition 3.7.8
of \cite{stichtenoth2009algebraic}, we have 
\[
g\left(C'\right)=pg\left(C\right)+\frac{p-1}{2}\left(-2+m_{x}+1\right).
\]

Let $b$ be large enough so that 
\[
e_{m}p^{b+1}\ge2\cdot4^{d}pg\left(C\right).
\]
Later, $b$ will correspond to the number of compositions of the Frobenius.

Note that increasing $m_{x}$ by 1 increases the expression $2\cdot4^{d}g\left(C'\right)$
by $\left(p-1\right)4^{d}$, so we can choose $m_{x}$ (since for
$m_{x}=1$, we already have $0\le e_{m}p^{b+1}-2\cdot4^{d}g\left(C'\right)$)
such that 
\begin{equation}
0\le e_{m}p^{b+1}-2\cdot4^{d}g\left(C'\right)<2\left(p-1\right)4^{d}\label{eq:lowerboundonem}
\end{equation}
(in case we get a multiple of $p$). Now, replace the map $C'\to C$
with the map $C'\to C'\to C$, where the first map $C'\to C'$ is
the $b$-fold Frobenius. 

Call this new map $k\colon C'\to C$ and write $h'\colon C'\to Y$
to be the composition $k\circ h$---recall that $h$ is a map sending
the curve $C$ to $Y$.
\begin{lem}
\label{lem:choiceofb}
\[
\frac{e_{m}p^{b+1}\left(n+1-d\right)}{m}<g\left(C'\right)-1.
\]
\begin{proof}
Since $m-2\cdot4^{d}\left(n+1-d\right)=1$, note that by assumption
on $p$, we have 
\begin{align*}
\frac{2\cdot4^{d}\left(p-1\right)\left(n+1-d\right)+m}{m-2\cdot4^{d}\left(n+1-d\right)} & <pg\left(C\right)\\
 & <pg\left(C\right)+\frac{p-1}{2}\left(m_{x}-1\right)\\
 & =g\left(C'\right).
\end{align*}
Rearranging this, we obtain 
\[
\frac{\left(2\cdot4^{d}g\left(C'\right)+2\left(p-1\right)4^{d}\right)\left(n+1-d\right)}{m}<g(C')-1,
\]
which implies 
\[
\frac{e_{m}p^{b+1}\left(n+1-d\right)}{m}<g\left(C'\right)-1.
\]
\end{proof}
\end{lem}

\begin{lem}
\label{lem:comparisonbounds}
\[
-h'^{*}K_{Y}\cdot C'+\left(1-g(C')\right)n'>-h'^{*}K_{X}|_{Y}\cdot C'+\left(1-g\left(C'\right)\right)(n-1).
\]
\begin{proof}
Observe that 
\[
\left(h'^{*}K_{Y}-h'^{*}K_{X}|_{Y}\right)\cdot C'<\frac{e_{m}p^{b+1}\left(n+1-d\right)}{m}<g\left(C'\right)-1
\]
by (\ref{eq:boundfromfujitaassump}) and Lemma \ref{lem:choiceofb}.
Rearranging, we have 
\[
\left(-h'^{*}K_{Y}+h'^{*}K_{X}|_{Y}\right)\cdot C'>1-g\left(C'\right)\ge\left(1-g(C')\right)\left(n-1-n'\right),
\]
since $g\left(C'\right)>1$ and $n-1>n'$. Rearranging again, we have
\[
-h'^{*}K_{Y}\cdot C'+\left(1-g(C')\right)n'>-h'^{*}K_{X}|_{Y}\cdot C'+\left(1-g\left(C'\right)\right)(n-1).
\]
\end{proof}
\end{lem}

Note that any morphism $C'\to Y$ gives rise to a morphism $C'\to V$
by composition with the resolution of singularities $Y\to V$. Moreover,
$Y\to V$ is birational, so for a sufficiently small open $U$, it
follows that 

\[
\dim_{[h']}\Mor\left(C',V\right)\ge\dim_{[h']}\Mor_{U}\left(C',Y\right).
\]

\begin{lem}
\label{lem:lowerbound}
\[
\dim_{[h']}\Mor\left(C',V\right)\ge-h'^{*}K_{Y}\cdot C'+\left(1-g(C')\right)n'.
\]
\begin{proof}
By the discussion above, it suffices to apply Lemma \ref{lem:lowerboundim}
to $\dim_{[h']}\Mor_{U}\left(C',Y\right)$. To check the condition
that $-h'^{*}K_{Y}\cdot C'+\left(1-g(C')\right)n'\ge n'$, note that
Lemma \ref{lem:comparisonbounds} tells us that 
\begin{align*}
-h'^{*}K_{Y}\cdot C'+\left(1-g(C')\right)n' & >-h'^{*}K_{X}|_{Y}\cdot C'+\left(1-g\left(C'\right)\right)(n-1)\\
 & =e_{m}p^{b+1}\left(n+1-d\right)+\left(1-g\left(C'\right)\right)(n-1)\\
 & \ge2\cdot4^{d}g\left(C'\right)\left(n+1-d\right)+\left(1-g\left(C'\right)\right)(n-1)\\
 & =g\left(C'\right)\left(2\cdot4^{d}\left(n+1-d\right)-n+1\right)+n-1\\
 & \ge n',
\end{align*}
where in the third line we use (\ref{eq:lowerboundonem}). 
\end{proof}
\end{lem}

\begin{lem}
\label{lem:upperbound}
\[
\dim_{[h']}\Mor\left(C',X\right)=-h'^{*}K_{X}|_{Y}\cdot C'+\left(1-g\left(C'\right)\right)(n-1).
\]
\begin{proof}
This follows immediately from Theorem \ref{thm:main}, since $h'$
has sufficiently large degree (using the condition $e_{m}p^{b+1}\ge2\cdot4^{d}g\left(C'\right)$,
which is at least $2^{d-1}\left(d-1\right)^{2}\left(3g+1\right)+d\left\lceil \frac{3g-1}{2}\right\rceil +\left\lfloor \frac{g-1}{2}\right\rfloor $).
\end{proof}
\end{lem}

Combining Lemmas \ref{lem:comparisonbounds}, \ref{lem:lowerbound},
and \ref{lem:upperbound}, it is clear that 
\[
\dim_{[h']}\Mor\left(C',V\right)>\dim_{[h']}\Mor\left(C',X\right),
\]
which is evidently a contradiction and hence completes the proof of
Corollary \ref{cor:anumber}.

\bibliographystyle{plainurl}
\bibliography{Higher_Genus_Circle}

\end{document}